\documentclass[11pt,oneside,reqno]{amsart}

\usepackage{amsmath,amssymb,amsfonts}
\usepackage[left=2.5cm,right=2.5cm,top=2.5cm,bottom=2.5cm]{geometry}
\usepackage{kpfonts}
\usepackage{mathtools}
\usepackage{bbm}
\usepackage{hyperref}
\usepackage{color}
\usepackage{xcolor}

\newtheorem{theorem}{Theorem}[section]
\newtheorem{lemma}[theorem]{Lemma}

\theoremstyle{definition}

\newtheorem{example}[theorem]{Example}
\newtheorem{prop}{Proposition}
\newtheorem{corollary}{Corollary}
\newtheorem{hypothesis}{Conjecture}
\newtheorem{rem}{Remark}

\theoremstyle{remark}

\numberwithin{equation}{section}

\renewcommand{\Re}{{\rm Re}}

\newcommand{\ii}{{\rm i}}
\newcommand{\ee}{{\rm e}}
\newcommand{\1}{\mathbbm{1}}

\renewcommand{\d}{{\rm d}}

\newcommand{\R}{\mathbb{R}}
\newcommand{\C}{\mathbb{C}}
\newcommand{\N}{\mathbb{N}}
\newcommand{\Q}{\mathbb{Q}}
\newcommand{\E}{\mathbb{E}}
\newcommand{\Z}{\mathbb{Z}}
\newcommand{\Za}{\mathbb{Z}^{\ast}}

\newcommand{\od}{\overset{{\rm d}}{=}}
\newcommand{\todistr}{\overset{{\rm d}}{\longrightarrow}}
\newcommand{\toas}{\overset{{\rm a.s.}}{\longrightarrow}}

\begin{document}

\title[Almost periodic stochastic processes]{Almost periodic stochastic processes with applications to analytic number theory}

\author{Alexander Iksanov}
\address{Faculty of Computer Science and Cybernetics, Taras Shevchenko National University of Kyiv, Kyiv 01601, Ukraine}
\email{iksan@univ.kiev.ua}

\author{Zakhar Kabluchko}

\address{Institute of Mathematical Stochastics, Department of Mathematics and Computer Science, University of M\"{u}n\-ster, Orl\'{e}ans-Ring 10, D-48149 M\"{u}n\-ster, Germany}
\email{zakhar.kabluchko@uni-muenster.de}

\author{Alexander Marynych}
\address{Faculty of Computer Science and Cybernetics, Taras Shevchenko National University of Kyiv, Kyiv 01601, Ukraine}
\curraddr{School of Mathematical Sciences, Queen Mary University of London, Mile End Road, London E14NS, United Kingdom}
\email{marynych@knu.ua,o.marynych@qmul.ac.uk}

%    General info

%\date{01/MARCH/2025}

\dedicatory{This paper is dedicated to the anniversaries of Professors A.~V.~Skorokhod, V.~S.~Korolyuk and Yu.~V.~Kozachenko}


\begin{abstract}
A classical fact of the theory of almost periodic functions is the existence of their asymptotic
distributions. In probabilistic terms, this means
that if $f$ is a Besicovitch almost periodic function and
$V$ is a random variable uniformly distributed on $[-1,1]$, then the random variables $f(L\cdot V)$ converge in distribution, as $L\to\infty$, to a proper %,
non-degenerate random variable. We prove
a functional extension of this result for the random processes $(f(L\cdot V+t))_{t\in\mathbb{R}}$ in the space of Besicovitch almost periodic functions, and also in the sense of weak convergence of finite-dimensional distributions.  We further investigate the properties of the limiting stationary process and demonstrate applications in analytic number theory by extending the one-dimensional results of [Limiting distributions of the classical error terms of prime number theory, Quart. J. Math. 65 (2014), 743--780] and earlier works.
\end{abstract}

\maketitle

\section{Introduction}

The theory of almost periodic functions, inspired by the seminal contribution of Bohr~\cite{Bohr}, has attracted enormous attention due to various applications in dynamical systems, ergodic theory, number theory and other fields. Since the appearance of Bohr's original definition of almost periodic functions, nowadays referred to as {\it uniformly almost periodic functions}, many generalizations have been introduced with the most % mostly
widespread being Stepanov's, Weyl's and Besicovitch's almost periodic functions. All these spaces of almost periodic functions can be obtained through the procedure of completion of the set of trigonometric polynomials
$$
\mathcal{T}:=\left\{\sum_{k=1}^{n}a_k\ee^{\ii \lambda_k t}\,:\,n\in\mathbb{N},\; a_k\in\mathbb{C},\; \lambda_k\in\mathbb{R},\; \lambda_i\neq\lambda_j\;\text{for}\;i\neq j\right\}
$$
with respect to various norms. For example, completion of $\mathcal{T}$ with respect to the uniform norm $\|f\|_{\infty}:=\sup_{t\in\mathbb{R}}|f(t)|$ %,
gives a closed subspace of the space $C_b(\mathbb{R},\mathbb{C})$ of continuous bounded complex-valued functions, which coincides with the space of Bohr's uniformly almost periodic functions. In what follows we denote this space by $U$. Completion of $\mathcal{T}$ with respect to the Marcinkiewicz seminorm, see~\eqref{eq:marcinkiewicz_norm} below, leads to the space $B_2$ of Besicovitch's almost periodic functions, a basic object for the present paper.

Fix a probability space $(\Omega,\mathcal{F},\mathbb{P})$. Given the space $U$, or any other space $S$ of almost periodic functions defined via the completion procedure described above, we introduce its Borel $\sigma$-algebra and define any measurable mapping from $\Omega$ to $S$ as an \emph{almost periodic stochastic process}. To the best of our knowledge, the earliest work on almost periodic stochastic processes dates back to the late 1930s, specifically Slutsky's paper~\cite{Slutsky:1938}, where the author investigated what can now be recognized as Besicovitch’s almost periodic processes. These were expressed as convergent (Fourier) series of generalized trigonometric polynomials with random coefficients. In~\cite{Hunt}, Hunt studied random Fourier transforms and derived fundamental sample path properties for Fourier series with independent, centered random coefficients, assuming that the series were square-\-summ\-able. As we shall see, such Fourier series represent certain functional averages of Besicovitch’s almost periodic functions. A comprehensive treatment of random Fourier series can be found in Chapter 5 of~\cite{Kahane:1985}. Alternative approaches to almost periodicity in the context of stationary stochastic processes are discussed in~\cite{Gladyshev:1963} and other sources. For an in-depth exploration of so called $p$-th mean almost periodic random processes, the monograph~\cite{Bezandry+Diagana:2011} serves as an excellent resource.

A motivation for an investigation of almost periodic processes comes from analytic number theory. Arithmetic functions are deterministic. Nevertheless, many of these admit limiting distributions (in a sense to be defined below). Knowing this it is natural to ask for functional limit theorems, that is, for the existence of limiting processes. As we shall demonstrate in Section~\ref{sec:examples}, the limiting processes for several classical arithmetic functions turn out to be stationary almost periodic processes with a discrete spectrum.

\section{A brief reminder on almost periodic functions}

For the purpose of the present paper we recall briefly a construction of the space $B_2$ of Besicovitch's almost periodic functions, which are our main tools. On this path we follow a comprehensive source~\cite{Corduneanu}, see Section 2.5 therein.

\subsection{Besicovitch almost periodic functions}

Let $L_2^{{\rm loc}}(\mathbb{R},\mathbb{C})$ denote the space of locally square-integrable measurable complex-valued functions defined on $\mathbb{R}$. The mapping $\|\cdot\|_{M_2}:L_2^{{\rm loc}}(\mathbb{R},\mathbb{C})\mapsto [0,+\infty]$ given by
\begin{equation}\label{eq:marcinkiewicz_norm}
\|f\|_{M_2}:=\limsup_{L\to+\infty}\left(\frac{1}{2L}\int_{-L}^L |f(t)|^2 {\rm d}t\right)^{1/2}
\end{equation}
is a semi-norm on the set $M_2(\mathbb{R},\mathbb{C}):=\{f\in L_2^{{\rm loc}}(\mathbb{R},\mathbb{C})\,:\,\|f\|_{M_2}<\infty\}$. The set $Z:=\{f\in M_2(\mathbb{R},\mathbb{C}):\|f\|_{M_2}=0\}$ is closed with respect to $\|\cdot\|_{M_2}$. Hence, the quotient
$$
\mathcal{M}_2(\mathbb{R},\mathbb{C}):=M_2(\mathbb{R},\mathbb{C})/Z
$$
is a normed vector space with respect to the norm $\|f+Z\|_{M_2}:=\|f\|_{M_2}$. This space is complete and known in the literature as a {\it Marcinkiewicz space}, see Propositions 2.18 and 2.19 in~\cite{Corduneanu}. Similarly to the $L_p$ spaces, with a slight abuse of notation, we shall think of elements in $\mathcal{M}_2(\mathbb{R},\mathbb{C})$ as functions even though these elements are in fact equivalence classes. Note that two functions in the same class may take different values everywhere on sets of positive measure and even on sets of infinite measure. For example, the function $t\mapsto \ee^{-|t|}$ belongs to $Z$ and is therefore equivalent to the zero function.

The Besicovitch space $B_2$ is defined as closure of $\mathcal{T}$ (or, more precisely, of $\mathcal{T}+Z$) in $\mathcal{M}_2(\mathbb{R},\mathbb{C})$. The space $B_2$ is complete but it is not separable. Indeed, $\{t\mapsto \ee^{\ii \lambda t}: \lambda\in\mathbb{R}\}\subset B_2$ is an uncountable set such that $\|\ee^{\ii \lambda_1 t}-\ee^{\ii \lambda_2 t}\|_{M_2}=\sqrt{2}$ if $\lambda_1\neq\lambda_2$. It is clear that the space $U$ of uniformly almost periodic functions is a dense (with respect to $\|\cdot\|_{M_2}$) subspace of $B_2$.

\subsection{Mean values and Fourier series}
Let $V_{-1,1}$ be a random variable defined on the probability space $(\Omega,\mathcal{F},\mathbb{P})$, which has 
the uniform distribution on $[-1,\,1]$. Put $V_{-L,L}:=L\cdot V_{-1,1}$, so that $V_{-L,L}$ is uniformly distributed on $[-L,L]$. It is known that any $f\in B_2$ possesses a mean value $M(f)$ defined by
\begin{equation}\label{eq:mean_value_def}
M(f):=\lim_{L\to+\infty}\frac{1}{2L}\int_{-L}^{L}f(t){\rm d}t=\lim_{L\to+\infty}\mathbb{E}[f(V_{-L,L})].
\end{equation}
According to Theorem 2.9 in~\cite{Akbary+Ng+Shahabi:2014}, see also Theorem 7 in~\cite{Duy:2013}, a more general fact holds true, namely, for any globally Lipschitz and globally bounded function $g:\mathbb{C}\to \mathbb{C}$, the limit
\begin{equation}\label{eq:gf_limit_mean_value}
\lim_{L\to+\infty}\mathbb{E}[g(f(V_{-L,L}))].
\end{equation}
exists. In other words, the following convergence in distribution holds
\begin{equation}\label{eq:one_dim_convergence}
f(V_{-L,L})~\todistr~\mathcal{M}_f(0),\quad L\to+\infty,
\end{equation}
where $\mathcal{M}_f(0)$ is a complex random variable with the distribution implicitly defined by
$$
\mathbb{E}[g(\mathcal{M}_f(0))]:=\lim_{L\to+\infty}\mathbb{E}[g(f(V_{-L,L}))].
$$
Since $U$ is a subset of $B_2$,~\eqref{eq:one_dim_convergence} obviously holds true also for $f\in U$.

Note that the aforementioned Theorem 2.9 in~\cite{Akbary+Ng+Shahabi:2014} is, in fact, an elementary consequence of the following observation. If $f\in B_2$, then $g\circ f\in B_2$ for any bounded globally Lipschitz function $g$ as follows immediately from the equivalent definition of the Besicovitch space $B_2$ using translation numbers, see p.~78 in the seminal Besicovitch book~\cite{Besicovitch}. Therefore, the limit in~\eqref{eq:gf_limit_mean_value} exists and is equal to $M(g\circ f)$.

Let $a\in\mathbb{R}$ be fixed and put $f(a;t):=f(t)\ee^{-\ii a t}$ for $t\in\mathbb{R}$. According to~\eqref{eq:mean_value_def}, the limit
$$
F_f(a):=M(f(a;t)):=\lim_{L\to+\infty}\frac{1}{2L}\int_{-L}^{L}f(t)\ee^{-\ii a t}{\rm d}t
$$
exists. This limit is non-zero for at most countably many $a\in\mathbb{R}$, see Proposition 4.1 in~\cite{Corduneanu}. The set
$$
\mathcal{S}(f):=\{a\in\mathbb{R}:\,F_f(a)\neq 0\}
$$
is called the Fourier spectrum of $f$, and the numbers $\{F_f(a)\,:\,a\in \mathcal{S}(f)\}$ are called the Fourier coefficients of $f$. The formal series $\sum_{a\in \mathcal{S}(f)}F_f(a)\ee^{\ii a t}$
is called the Fourier series of $f$ and the pairs $\{(a,F_f(a))\,:\,a\in \mathcal{S}(f)\}$ uniquely define $f$ (or, more precisely, $f+Z$), see p.~104 in~\cite{Corduneanu}. The notation
$$
f(t)~\sim~\sum_{a\in \mathcal{S}(f)}F_f(a)\ee^{\ii a t}
$$
is used to denote that $f$ has the formal Fourier series $\sum_{a\in \mathcal{S}(f)}F_f(a)\ee^{\ii a t}$. The Fourier coefficients of $f\in B_2$ satisfy a Parseval equality
\begin{equation}\label{eq:square_suumable}
\sum_{a\in\mathcal{S}(f)}|F_f(a)|^2=M(|f|^2)=\|f\|^2_{M_2}<\infty.
\end{equation}
It can be checked that the Fourier spectrum of a trigonometric polynomial $\sum_{k=1}^{n}a_k\ee^{\ii \lambda_k t}$ with $a_k\neq 0$, $k=1,\ldots,n$, is $\{\lambda_1,\ldots,\lambda_n\}$.

\subsection{Besicovitch functions with a restricted spectrum}
For a subset $A\subset\mathbb{R}$, put
$$
\mathcal{T}(A):=\{f\in\mathcal{T}\,:\,\mathcal{S}(f)\subseteq A\}=\left\{\sum_{k=1}^{n}a_k\ee^{\ii \lambda_k t}\,:\,n\in\mathbb{N},\;a_k\in\mathbb{C},\;\lambda_k\in A,\;\lambda_i\neq\lambda_j\;\text{for}\;i\neq j\right\}
$$
and let $B_2(A)$ be the closure of $\mathcal{T}(A)$ in $B_2$ with respect to $\|\cdot\|_{M_2}$.

\begin{prop}
For every $A\subseteq\mathbb{R}$, $B_2(A)=\{f\in B_2:\,\mathcal{S}(f)\subseteq A\}$.
\end{prop}
\begin{proof}
We first show that $f\in B_2(A)$ and $a\in \mathcal{S}(f)$ imply $a\in A$. We argue by contradiction and suppose that $a\notin A$. Take a sequence of trigonometric polynomials $(P_m)_{m\in\mathbb{N}}\subset \mathcal{T}(A)$ such that $\|f-P_m\|_{M_2}\to 0$, as $m\to\infty$. Then, for every $m\in\mathbb{N}$,
\begin{equation}\label{eq:prop1_proof1}
\lim_{L\to+\infty}\frac{1}{2L}\int_{-L}^L f(a;t){\rm d}t=\lim_{L\to+\infty}\frac{1}{2L}\int_{-L}^L f(t)\ee^{-\ii a t}{\rm d}t=\lim_{L\to+\infty}\frac{1}{2L}\int_{-L}^{L} (f(t)-P_m(t))\ee^{-\ii a t}{\rm d}t,
\end{equation}
but, by the Cauchy--Schwartz inequality,
$$
\lim_{L\to+\infty}\frac{1}{2L}\left|\int_{-L}^L (f(t)-P_m(t))\ee^{-\ii a t}{\rm d}t\right|\leq \limsup_{L\to+\infty} \left(\frac{1}{2L}\int_{-L}^L |(f(t)-P_m(t)|^2{\rm d}t\right)^{1/2}=\|f-P_m\|_{M_2}.
$$
Sending $m\to\infty$ shows that the left-hand side of~\eqref{eq:prop1_proof1} is equal to zero which contradicts $a\in \mathcal{S}(f)$.

In the other direction, suppose that $f\in B_2$ and $\mathcal{S}(f)\subseteq A$. It is known, see Theorem II on p.~105 in~\cite{Besicovitch}, that the Bochner--Fej\'{e}r sequence $(\sigma_m(f))_{m\in\mathbb{N}}$ of trigonometric polynomials constructed from the function $f$ belongs to $\mathcal{T}(A)$ and converges to $f$ with respect to the $\|\cdot\|_{M_2}$ norm. Thus, $f\in B_2(A)$.
\end{proof}
If the set $A=\{\lambda_1,\ldots,\lambda_n\}$ is finite, the mapping
$$
B_2(A)\ni f\mapsto (F_f(\lambda_1),\ldots,F_f(\lambda_n))\in \mathbb{C}^n
$$
defines an isometry between $B_2(A)$ and $\mathbb{C}^{n}$ viewed as inner product spaces (notation $B_2(A) \simeq \mathbb{C}^{n}$). Here, $B_2(A)$ is endowed with the inner product
\begin{equation}\label{eq:inner_product_B_2(A)}
\langle f,g\rangle_2:=\limsup_{L\to+\infty}\frac{1}{2L}\int_{-L}^{L}f(t)\overline{g(t)}{\rm d}t,\quad f,g\in B_2(A).
\end{equation}
Indeed,
$$
\langle f,g\rangle_2=\limsup_{L\to+\infty}\frac{1}{2L}\int_{-L}^{L}\sum_{i,j=1}^{n}F_f(\lambda_i)\overline{F_g(\lambda_j)}\ee^{\ii(\lambda_i-\lambda_j)t}{\rm d}t=\sum_{i=1}^{n}F_f(\lambda_i)\overline{F_g(\lambda_i)}.
$$
Similarly, if the set $A$ is countably infinite,~\eqref{eq:square_suumable} implies that $B_2(A)$ endowed with the inner product~\eqref{eq:inner_product_B_2(A)} is isometric to the infinite-dimensional (separable) Hilbert space $\ell_2(\mathbb{C})$ of square-\-summ\-able sequences over $\mathbb{C}$. A particularly important case is obtained by letting $A$ be the spectrum $\mathcal{S}(f)$ of a {\em fixed} Besicovitch almost periodic function $f$. Summarizing, we arrive at the following result.
\begin{prop}\label{prop:isometry}
Let $f\in B_2$ be a fixed Besicovitch almost periodic function. The space $B_2(\mathcal{S}(f))$ is a complete separable Hilbert space over $\mathbb{C}$ with respect to the inner product~\eqref{eq:inner_product_B_2(A)}. If $|\mathcal{S}(f)|<\infty$, then $B_2(\mathcal{S}(f))\simeq \mathbb{C}^{|\mathcal{S}(f)|}$, whereas if $\mathcal{S}(f)$ is countably infinite, then $B_2(\mathcal{S}(f))\simeq \ell_{2}(\mathbb{C})$. In both cases the isometry $\kappa_f$ is given by
$$
B_2(\mathcal{S}(f))\ni g\longmapsto \kappa_f(g):=(F_g(\lambda))_{\lambda\in \mathcal{S}(f)}.
$$
\end{prop}

%In what follows we shall always exclude the simple case when $\mathcal{S}(f)$ is finite.

\section{Almost periodic stochastic processes}\label{sec:ap_processes}

Let $f\in B_2$ be a fixed Besicovitch almost periodic function. A stochastic process $(f(V_{-L,L}(\omega)+t))_{t\in\R}$ is a translation of $f$ by a random number $V_{-L,L}$ chosen uniformly from $[-L,L]$. Here, $\omega \in \Omega$ denotes an element of the probability space $(\Omega,\mathcal{F},\mathbb{P})$ on which the random variables $V_{-L,L}$ are defined. It is clear that $f(V_{-L,L}(\omega)+\cdot)\in B_2$ and the spectrum of $f(V_{-L,L}(\omega)+\cdot)$ coincides with the spectrum of $f(\cdot)$, for every fixed $\omega\in\Omega$. The latter follows from the equality\footnote{This argument also shows that arbitrary translations (random or deterministic) do not change the spectrum.}
\begin{equation}\label{eq:Fourier_coefficient_shifted}
F_{f(V_{-L,L}+\cdot)}(a)=\limsup_{T\to+\infty}\frac{1}{2T}\int_{-T}^{T}f(V_{-L,L}+t)\ee^{-\ii t a}{\rm d}t=\ee^{\ii a V_{-L,L}}F_f(a).
\end{equation}
Thus, for every $\omega\in\Omega$, $f(V_{-L,L}(\omega)+\cdot)$ is an element of $B_2(\mathcal{S}(f))$. Let $\mathcal{F}_{\mathcal{S}(f)}$ be the Borel sigma-algebra on $B_2(\mathcal{S}(f))$. The mapping
$$
\mathbb{R}\times\Omega\ni (t,\omega)\mapsto f(V_{-L,L}(\omega)+t)\in\mathbb{C} %,\quad t\in\mathbb{R},\quad \omega\in\Omega,
$$
defines a random element taking values in the measurable space $(B_2(\mathcal{S}(f)),\mathcal{F}_{\mathcal{S}(f)})$.

A simpler approach could have involved defining $t\mapsto f(V_{-L,L}+t)$ as a random element taking values in $B_2$. However, the non-separability of $B_2$ complicates the analysis, as standard probabilistic tools (tightness of a probability measure, Prohorov's and Skorokhod's representation theorems, etc.) are not readily applicable in non-separable spaces.

\subsection{Convergence in distribution on \texorpdfstring{$B_2$}{B2}}

Theorem~\ref{thm:conv_in_B2} below is a functional version of~\eqref{eq:one_dim_convergence} in the subspace $B_2(\mathcal{S}(f))$ of $B_2$. It is the starting observation for the present paper. Even though the result is rather simple, we have not been able to locate it in the literature.

Let $\mathbb{T}$ be the unit circle in $\mathbb{C}$ regarded as a topological group and
$$
\mathbb{T}_f:=\bigtimes_{\lambda\in\mathcal{S}(f)}\mathbb{T}_{\lambda}
$$
%be
the torus obtained by taking (either finite if $|\mathcal{S}(f)|<\infty$, or infinite if $|\mathcal{S}(f)|=\infty$) product of copies $\mathbb{T}_{\lambda}$ of $\mathbb{T}$ endowed with the product topology. The collection $(\ee^{\ii \lambda V_{-L,L}})_{\lambda\in\mathcal{S}(f)}$ is a random element taking values in $\mathbb{T}_f$. According to Theorem 3.1 in~\cite{Duy:2012}, it converges in distribution on $\mathbb{T}_f$, as $L\to\infty$, to a random variable, to be denoted by $\mathcal{V}=(\mathcal{V}_{\lambda})_{\lambda\in\mathcal{S}(f)}$, whose distribution can be uniquely characterized by the Fourier transform on the Pontryagin dual group $\widehat{\mathbb{T}_f}\simeq \mathbb{Z}^{|\mathcal{S}(f)|}$, see the proof of Theorem~\ref{cor:conv_in_B2_one_sided} below and, particularly, formula~\eqref{eq:fourier_transform_of_V}.

If the collection $\mathcal{S}(f)\subset\mathbb{R}$ is linearly independent over $\mathbb{Q}$, then $(\mathcal{V}_{\lambda})_{\lambda\in\mathcal{S}(f)}$ has the same distribution as a collection $(\mathcal{U}_{\lambda})_{\lambda\in\mathcal{S}(f)}$ of mutually independent and identically distributed random variables, each with the uniform distribution on $\mathbb{T}$, see~\cite{Laurinchikas4,Laurinchikas5,Laurinchikas7,Laurinchikas2,Laurinchikas6,Laurinchikas3} and also~\eqref{eq:Kronecker-Weyl} below. A bit more general, yet useful for our purposes, is the situation when $\mathcal{S}(f)$ possesses a decomposition $\mathcal{S}(f)=\mathcal{C}\cup(-\mathcal{C})\cup \mathcal{R}\cup\{0\}$ or $\mathcal{S}(f)=\mathcal{C}\cup(-\mathcal{C})\cup \mathcal{R}$, where the sets on the right-hand side are pairwise disjoint and $\mathcal{C}\cup\mathcal{R}$ (or, equivalently $(-\mathcal{C})\cup\mathcal{R}$) is linearly independent over $\mathbb{Q}$. The latter, in particular, implies that $\mathcal{R}\cap(-\mathcal{R})=\varnothing$. In other words, $\mathcal{S}(f)$ may now contain $0$ and the opposite pairs $-\lambda,\lambda$ which are obviously linearly dependent over $\mathbb{Q}$. The union $\mathcal{C}\cup(-\mathcal{C})$ contains all such pairs, whereas $\mathcal{R}$ consists of all $\lambda\in\mathcal{S}(f)$ such that $-\lambda\notin\mathcal{S}(f)$. In the most general situation of the described case, the collection $\mathcal{V}=(\mathcal{V}_{\lambda})_{\lambda\in\mathcal{S}(f)}$ is comprised of four subcollections $(\overline{\mathcal{U}_{-\lambda}})_{\lambda\in(-\mathcal{C})}$, $(\mathcal{U}_{\lambda})_{\lambda\in\mathcal{C}}$, $(\mathcal{U}_{\lambda})_{\lambda\in\mathcal{R}}$ and $\mathcal{V}_0=1$, with $(\mathcal{U}_{\lambda})_{\lambda\in\mathcal{C}\cup\mathcal{R}}$ being mutually independent uniformly distributed on $\mathbb{T}$ random variables, and $\overline{z}$ is the complex conjugate of $z$.  A typical example is $\mathcal{S}(f)=\mathcal{C}\cup(-\mathcal{C})$, for $\mathcal{C}$ such that $\mathcal{C}$ is linearly independent over $\mathbb{Q}$, where the spectrum is symmetric around the origin and does not contain $0$.

\begin{theorem}\label{thm:conv_in_B2}
Suppose that $f\in B_2$. Let $\mathcal{S}(f)$ be the spectrum of $f$ and
$$
f(t)~\sim~\sum_{\lambda\in\mathcal{S}(f)}f_{\lambda}\ee^{\ii \lambda t},\quad f_{\lambda}=F_f(\lambda)=M(f(t)\ee^{-\ii \lambda t}),\quad \lambda\in\mathcal{S}(f),
$$
be the formal Fourier series of $f$. Then, as $L\to+\infty$, the stochastic processes $f(V_{-L,L}+\cdot)$ converge in distribution on the space $B_2(\mathcal{S}(f))$ to a random element $\mathcal{M}_{f}\in B_2(\mathcal{S}(f))\subset B_2$ uniquely defined by the formal Fourier series
$$
\mathcal{M}_{f}(t)~\sim~\sum_{\lambda\in\mathcal{S}(f)}\mathcal{V}_{\lambda}f_{\lambda}\ee^{\ii \lambda t}.
$$
\end{theorem}
\begin{proof}
In view of~\eqref{eq:Fourier_coefficient_shifted} the formal Fourier series of $f(V_{-L,L}+\cdot)$ is given by
$$
f(V_{-L,L}+t)~\sim~\sum_{\lambda\in\mathcal{S}(f)}\ee^{\ii \lambda V_{-L,L}}f_{\lambda}\ee^{\ii \lambda t}.
$$
In the subsequent proof we assume that $|\mathcal{S}(f)|=\infty$. A proof in the case of finite $\mathcal{S}(f)$ requires only minor modifications and is, therefore, omitted.

By the isometry established in Proposition~\ref{prop:isometry} it suffices to check that $\ell_2(\mathbb{C})$-valued random variables $(\ee^{\ii \lambda V_{-L,L}}f_{\lambda})_{\lambda\in\mathcal{S}(f)}$ converge in distribution, as $L\to+\infty$, to $(\mathcal{V}_{\lambda}f_{\lambda})_{\lambda\in\mathcal{S}(f)}$. By the Skorokhod representation theorem, retaining the original notation for the versions, we can and do assume that, for each $\lambda\in\mathcal{S}(f)$,
$$
\ee^{\ii \lambda V_{-L,L}}~\to~\mathcal{V}_{\lambda}\quad\text{a.s.}
$$
Using that $\sum_{\lambda\in\mathcal{S}(f)}|f_{\lambda}|^2<\infty$, we obtain, by the Lebesgue dominated convergence theorem, that
\begin{multline*}
\lim_{L\to+\infty}\sum_{\lambda\in\mathcal{S}(f)}|\ee^{\ii \lambda V_{-L,L}}f_{\lambda}-\mathcal{V}_{\lambda}f_{\lambda}|^2=\lim_{L\to+\infty}\sum_{\lambda\in\mathcal{S}(f)}|f_{\lambda}|^2 |\ee^{\ii \lambda V_{-L,L}}-\mathcal{V}_{\lambda}|^2\\
=\sum_{\lambda\in\mathcal{S}(f)}|f_{\lambda}|^2 \lim_{L\to+\infty}|\ee^{\ii \lambda V_{-L,L}}-\mathcal{V}_{\lambda}|^2=0\quad\text{a.s.}
\end{multline*}
The proof is complete.
\end{proof}
On many occasions it is more convenient to work with a slightly different, one-sided averages defined by $t\mapsto f(V_{0,L}+t)$, where $V_{0,L}=L\cdot V_{[0,1]}$ and  $V_{[0,1]}$ is a random variable on $(\Omega,\mathcal{F},\mathbb{P})$ with the uniform distribution on $[0,1]$. It turns out that this type of averaging leads to the same limiting process $\mathcal{M}_{f}$. In order to check this statement, we % shall
first prove a more general version of Theorem~\ref{thm:conv_in_B2}.

\begin{theorem}\label{cor:conv_in_B2_one_sided}
Under the assumptions of Theorem~\ref{thm:conv_in_B2}, there is joint convergence in distribution
$$
\left(f(V_{-L,L}+\cdot),V_{-1,1}\right)~\Longrightarrow~(\mathcal{M}_{f},\mathcal{U}),\quad L\to+\infty,
$$
on the product space $B_2(\mathcal{S}(f))\times \mathbb{R}$, where $\mathcal{U}$ has the uniform distribution on $[-1,1]$ and is independent of $\mathcal{M}_{f}$.
\end{theorem}

Note that the non-obvious part of the statement is {\em independence} of the components on the right-hand side!

\begin{proof}
It suffices to check that
\begin{equation}\label{cor:conv_in_B2_one_sided_proof1}
\left((\ee^{\ii \lambda V_{-L,L}})_{\lambda\in\mathcal{S}(f)},V_{-1,1}\right)~\Longrightarrow~(\mathcal{V},\mathcal{U}),\quad L\to\infty,
\end{equation}
on the product space $\mathbb{T}_f\times\mathbb{R}$ with $\mathcal{V}$ and $\mathcal{U}$ independent.
The subsequent argument is an adaptation of the proof of Theorem 3.1 in~\cite{Duy:2012}. Without loss of generality, assume that $|\mathcal{S}(f)|=\infty$.
The Pontryagin dual group of $\mathbb{T}_f\times \mathbb{R}$ is isomorphic to $\left(\oplus_{\lambda\in\mathcal{S}(f)}\mathbb{Z}_{\lambda}\right)\oplus \mathbb{R}$, where $\mathbb{Z}_{\lambda}$, for $\lambda\in\mathcal{S}(f)$, are copies of $\mathbb{Z}$. Let ${\bf k}=(k_{\lambda})_{\lambda\in\mathcal{S}(f)}\in\oplus_{\lambda\in\mathcal{S}(f)}\mathbb{Z}_{\lambda}$ be such that only finitely many integers $k_{\lambda}$ are non‐zero and let $t\in\mathbb{R}$ be a fixed real number. The pair $({\bf k},t)$ (identified with the character on $\mathbb{T}_f\times \mathbb{R}$) acts on $\mathbb{T}_f\times\mathbb{R}$ by the rule
$$
\mathbb{T}_f\times\mathbb{R}\ni ((v_{\lambda})_{\lambda\in\mathcal{S}(f)},u)\longmapsto \ee^{\ii t u}\prod_{\lambda\in\mathcal{S}(f)}v_{\lambda}^{k_{\lambda}}.
$$
The Fourier transform $g_L(({\bf k},t))$ of the left-hand side of~\eqref{cor:conv_in_B2_one_sided_proof1} is equal to
\begin{multline*}
g_L(({\bf k},t))=\mathbb{E}\left[\ee^{\ii t V_{-1,1}}\prod_{\lambda\in\mathcal{S}(f)}\ee^{\ii \lambda k_{\lambda} V_{-L,L}}\right]=\frac{1}{2L}\int_{-L}^{L}\ee^{\ii x(t/L+\sum_{\lambda\in\mathcal{S}(f)}\lambda k_{\lambda})}{\rm d}
x\\
=\frac{1}{2L}\int_{-L}^{L}\cos\left(x\left(t/L+\sum_{\lambda\in\mathcal{S}(f)}\lambda k_{\lambda}\right)\right){\rm d}x.
\end{multline*}
As $L\to+\infty$, the right-hand side converges to the product $g({\bf k})\sin(t)/t$ of the (usual) Fourier transform $\sin t/t$ of $\mathcal{U}$ and the Fourier transform $g({\bf k})$ of $\mathcal{V}$ given by
\begin{equation}\label{eq:fourier_transform_of_V}
g({\bf k})=\begin{cases}
1,&\text{if }\sum_{\lambda\in\mathcal{S}(f)}\lambda k_{\lambda}=0,\\
0,&\text{otherwise}.\end{cases}
\end{equation}
The proof is complete.
\end{proof}

Below is the promised limit theorem for the one-sided averages $t\mapsto f(V_{0,L}+t)$.
\begin{corollary}\label{cor:one_sided_convergence_in_B2}
Under the assumptions of Theorem~\ref{thm:conv_in_B2}, the stochastic processes $f(V_{0,L}+\cdot)$ converge in distribution on the space $B_2(\mathcal{S}(f))$ to $\mathcal{M}_{f}$, as $L\to+\infty$.
\end{corollary}
\begin{proof}
Note that the distribution of $f(V_{0,L}+\cdot)\in B_2(\mathcal{S}(f))$ coincides with the conditional distribution of $f(V_{-L,L}+\cdot)$ given $V_{-1,1}\geq 0$. Therefore, for every $A\in\mathcal{F}_{\mathcal{S}(f)}$ satisfying
$\mathbb{P}\{\mathcal{M}_{f}\in\partial A\}=0$,
$$
\mathbb{P}\{f(V_{0,L}+\cdot)\in A\}=\frac{\mathbb{P}\{f(V_{-L,L}+\cdot)\in A, V_{-1,1}\geq 0\}}{\mathbb{P}\{V_{-1,1}\geq 0\}}~\to~2\mathbb{P}\{\mathcal{M}_{f}\in A, \mathcal{U}\geq 0\}=\mathbb{P}\{\mathcal{M}_{f}\in A\},
$$
as $L\to+\infty$, having utilized independence of $\mathcal{M}_{f}$ and $\mathcal{U}$ for the last equality.
\end{proof}

\subsection{The limit process for randomly translated almost periodic function}
From now on we % shall
always assume that $f\in B_2\setminus\{0\}$ satisfies:
\begin{equation}\label{eq:lin_independence_gen}
\begin{aligned}
&\mathcal{S}(f)\text{ does not have finite accumulation points and is decomposed into pairwise}\\
&\text{ disjoint sets }\mathcal{S}(f)=\mathcal{C}\cup(-\mathcal{C})\cup\mathcal{R},\text{ where }\mathcal{C}\cup\mathcal{R}\text{ is linearly independent over }\mathbb{Q}.
\end{aligned}
\end{equation}
In particular, we assume that $0\notin\mathcal{S}(f)$. This does not reduce generality: if $0\in\mathcal{S}(f)$ and $f_0$ is the corresponding Fourier coefficient, then $g(t):=f(t)-f_0$ satisfies $\mathcal{S}(g)=\mathcal{S}(f)\setminus\{0\}$. As has already been discussed, the components of the decomposition in~\eqref{eq:lin_independence_gen} necessarily satisfy $\mathcal{R}\cap(-\mathcal{R})=\varnothing$ and, without loss of generality, we can and do assume that $\mathcal{C}\subset [0,+\infty)$. Since % we assumed that
$\mathcal{S}(f)$ does not have finite accumulation points by assumption, we can enumerate the points in $\mathcal{C}=\{\lambda_{k}^{(c)}\,:\,k\geq 1\}$ in the increasing order
\begin{equation}\label{eq:lambda_c_enumeration}
0 < \lambda_1^{(c)} < \lambda_2^{(c)}<\cdots<+\infty.
\end{equation}
The same applies to $\mathcal{R}$ but the corresponding sequence is, in general, two-sided $\mathcal{R}=\{\lambda_{k}^{(r)}\,:\,k\in\Za\}$ with $\Za:=\mathbb{Z}\setminus\{0\}$ and
\begin{equation}\label{eq:lambda_r_enumeration}
-\infty <\cdots <\lambda_{-2}^{(r)} < \lambda_{-1}^{(r)}<0<\lambda_{1}^{(r)}<\lambda_{2}^{(r)}<\cdots<+\infty.
\end{equation}
Naturally, the sequence given in~\eqref{eq:lambda_c_enumeration} is allowed to be finite and the sequence given in~\eqref{eq:lambda_r_enumeration} is allowed to be finite or one-sided infinite.

With the introduced notation, the Fourier series of $f$ and $\mathcal{M}_f$ can be written now as
\begin{align*}
&f(t)~\sim~\sum_{k\geq 1}\left(f_{\lambda^{(c)}_k}\ee^{\ii \lambda_k^{(c)} t}+f_{-\lambda^{(c)}_k}\ee^{-\ii \lambda_k^{(c)} t}\right)+\sum_{k\in\Za}f_{\lambda_k^{(r)}}\ee^{\ii \lambda_k^{(r)} t},\\
&\mathcal{M}_{f}(t)~\sim~\sum_{k\geq 1}\left(\mathcal{U}_{k}^{(c)}f_{\lambda^{(c)}_k}\ee^{\ii \lambda_k^{(c)} t}+\overline{\mathcal{U}_{k}^{(c)}}f_{-\lambda^{(c)}_k}\ee^{-\ii \lambda_k^{(c)} t}\right)+\sum_{k\in\Za}\mathcal{U}_{k}^{(r)}f_{\lambda_k^{(r)}}\ee^{\ii \lambda_k^{(r)} t},
\end{align*}
respectively, where $(\mathcal{U}_{k}^{(c)})_{k\geq 1}$ and $(\mathcal{U}_{k}^{(r)})_{k\in\Za}$ are independent sequences of independent identically distributed random variables, each with the uniform distribution on the unit circle $\mathbb{T}$.

Under~\eqref{eq:lin_independence_gen}, the formal Fourier series of $\mathcal{M}_{f}$ is the sum of two independent random series, denoted in what follows by $\mathcal{M}_f^{(c)}(t)$ and $\mathcal{M}_f^{(r)}(t)$, each of which is itself a sum of mutually independent random variables. By Kolmogorov's two series theorem they both converge a.s. for every fixed $t\in\mathbb{R}$. A less obvious fact is that they both converge with probability one {\em almost everywhere} in $t$ with respect to Lebesgue measure on $\mathbb{R}$, see p.~296-297 in~\cite{Jessen} and, therefore, define a random function $\mathcal{M}_{f}$. By Theorem 1 in~\cite{Hunt}, $\mathcal{M}_{f}\in \bigcap_{p\geq 1} L_p^{{\rm loc}}(\mathbb{R},\mathbb{C})$ with probability one. Thus, we can write
\begin{align*}
\mathcal{M}_{f}^{(c)}(t)~=~\sum_{k\geq 1}\left(\mathcal{U}_{k}^{(c)}f_{\lambda^{(c)}_k}\ee^{\ii \lambda_k^{(c)} t}+\overline{\mathcal{U}_{k}^{(c)}}f_{-\lambda^{(c)}_k}\ee^{-\ii \lambda_k^{(c)} t}\right),\quad\mathcal{M}_{f}^{(r)}(t)~=~\sum_{k\in\Za}\mathcal{U}_{k}^{(r)}f_{\lambda_k^{(r)}}\ee^{\ii \lambda_k^{(r)} t},\\
\end{align*}
and
$$
\mathcal{M}_{f}(t)=\mathcal{M}_{f}^{(c)}(t)+\mathcal{M}_{f}^{(r)}(t),
$$
with $=$ instead of $\sim$ everywhere, and regard $\mathcal{M}_{f}^{(c)}$, $\mathcal{M}_{f}^{(r)}$ and $\mathcal{M}_{f}$ as random functions on $\mathbb{R}$, rather than just elements of $B_2$.

%A curious observation is that the Fourier series of $f$ itself does not necessarily converges to $f$ pointwise!

\begin{theorem}\label{thm:ergodic}
Assume that $f\in B_2$ satisfies~\eqref{eq:lin_independence_gen}. The random process $(\mathcal{M}_{f}(t))_{t\in\mathbb{R}}$ is a centered stationary process with the covariance
\begin{equation}\label{eq:covariance_gen}
K_{\mathcal{M}_{f}}(s):=\E [\mathcal{M}_{f}(s) \overline{\mathcal{M}_{f}(0)}]=\sum_{\lambda\in\mathcal{S}(f)}|f_{\lambda}|^2\ee^{\ii \lambda s},\;\; \E [\mathcal{M}_{f}(s) \mathcal{M}_{f}(0)]=\sum_{k\geq 1}f_{\lambda_k^{(c)}}f_{-\lambda_k^{(c)}}\ee^{\ii \lambda_k^{(c)} s},\;\; s\in\mathbb{R}.
\end{equation}
Moreover, the process $\mathcal{M}_{f}$ is ergodic in the usual sense: with $\mathfrak{m}$ denoting the distribution of $\mathcal{M}_{f}(0)$,
$$
\lim_{T\to+\infty}\frac{1}{T}\int_0^T \phi(\mathcal{M}_{f}(t)){\rm d}t=
\E \phi(\mathcal{M}_{f}(0))=\int_{\mathbb{C}}\phi(x)\mathfrak{m}({\rm d}x)\quad\text{a.s.}
$$
for every $\phi\in L_1(\mathbb{C},\mathfrak{m})$. The process $\mathcal{M}_{f}$ is not mixing.
\end{theorem}
\begin{rem}
Formula~\eqref{eq:covariance_gen} together with $\sum_{\lambda\in\mathcal{S}(f)}|f_{\lambda}|^2<\infty$ demonstrates that $K_{\mathcal{M}_{f}}(s)$ is a uniform over $s\in\mathbb{R}$ limit of generalized trigonometric polynomials. Thus, $K_{\mathcal{M}_{f}}$ is Bohr's uniformly almost periodic function.
\end{rem}
\begin{proof}
Formulae~\eqref{eq:covariance_gen} follow by direct calculations using that $\E|\mathcal{U}_{\lambda}|^2=1$ and $\E \mathcal{U}_{\lambda}=\E \mathcal{U}_{\lambda}^2=\E \overline{\mathcal{U}_{\lambda}^2}=0$ for every $\lambda\in\mathcal{C}\cup\mathcal{R}$.

It suffices to justify the remaining claims separately for $\mathcal{M}_{f}^{(c)}$ and $\mathcal{M}_{f}^{(r)}$ and then use independence. We only do this for $\mathcal{M}_{f}^{(c)}$, leaving the easier case of $\mathcal{M}_{f}^{(r)}$ to the reader. We first check stationarity. By rotational invariance and independence,
$$
(\mathcal{U}_{k}^{(c)}\ee^{\ii b_{k}})_{k\geq 1}~\od~(\mathcal{U}^{(c)}_{k})_{k\geq 1}
$$
for any collection of real numbers $(b_{k})_{k\geq 1}$. Thus, for every fixed $s\in\R$, % it holds
\begin{multline*}
(\mathcal{M}_{f}^{(c)}(t+s))_{t\in\R}=\left(\sum_{k\geq 1}\left(\mathcal{U}_{k}^{(c)}\ee^{\ii\lambda_k^{(c)} s}f_{\lambda_k^{(c)}}\ee^{\ii \lambda_k^{(c)} t}+\overline{\mathcal{U}_{k}^{(c)}\ee^{\ii\lambda_{k}^{(c)} s}}f_{-\lambda_k^{(c)}}\ee^{-\ii \lambda_k^{(c)} t}\right)\right)_{t\in\R}\\
\od~\left(\sum_{k\geq 1}\left(\mathcal{U}_{k}^{(c)}f_{\lambda_k^{(c)}}\ee^{\ii \lambda_k^{(c)} t}+\overline{\mathcal{U}_{k}}f_{-\lambda_k^{(c)}}\ee^{-\ii \lambda_k^{(c)} t}\right)\right)_{t\in\R}=(\mathcal{M}_{f}^{(c)}(t))_{t\in\R}.
\end{multline*}
In order to prove ergodicity of $\mathcal{M}_f^{(c)}$, recall that, under % assumption
~\eqref{eq:lin_independence_gen}, $(\mathcal{U}_k^{(c)})_{k\geq 1}$ is distributed according to the Haar measure on the compact Abelian group
$$
\mathbb{T}_f^{(c)}:=\bigtimes_{\lambda\in\mathcal{C}}\mathbb{T}_{\lambda}=\bigtimes_{k\geq 1}\mathbb{T}_{\lambda_k^{(c)}}.
$$
Denote this Haar measure by $\mathfrak{H}$. The triple $(\mathbb{T}_f^{(c)},\mathcal{B}(\mathbb{T}_f^{(c)}),\mathfrak{H})$ is a probability space and, on this space, $(\mathcal{U}_k^{(c)})_{k\geq 1}$ may be defined as the identity mapping. The mappings
$$
T^f_t:\mathbb{T}_f^{(c)}\mapsto \mathbb{T}_f^{(c)},\quad T^f_t((z_{\lambda_k^{(c)}})_{k\geq 1}):=(\ee^{\ii \lambda_k^{(c)} t}z_{\lambda_k^{(c)}})_{k\geq 1}, \quad t\in\mathbb{R}
$$
preserve the measure $\mathfrak{H}$ and form a flow $(T^f_t)_{t\in\mathbb{R}}$ which satisfies $T^f_0={\rm Id}$, $T^f_{t+s}=T^f_t \circ T^f_s$, $t,s\in\mathbb{R}$. Note that each $T^f_t$ is an infinite direct product of irrational rotations. Let $g_{f}:\mathbb{T}_f^{(c)}\to\mathbb{C}$ be a function defined by
$$
g_{f}((z_{\lambda_k^{(c)}})_{k\geq 1}=\sum_{k\geq 1}\big(f_{\lambda_k^{(c)}} z_{\lambda_k^{(c)}}+f_{-\lambda_k^{(c)}}\overline{z_{\lambda_k^{(c)}}}\big).
$$
In view of % the relation
$\sum_{k\geq 1}(|f_{\lambda_k^{(c)}}|^2+|f_{-\lambda_k^{(c)}}|^2)\leq\sum_{\lambda\in\mathcal{S}(f)}|f_{\lambda}|^2<\infty$, the mapping $g_{f}$ is well-defined for $\mathfrak{H}$-almost all $(z_{\lambda})_{\lambda\in\mathcal{C}}\in \mathbb{T}_f^{(c)}$ (for $\mathfrak{H}$-almost all realizations ${\bf z}:=(z_{\lambda_k^{(c)}})_{k\geq 1}$ of $(\mathcal{U}_k^{(c)})_{k\geq 1}$) by Kolmogorov's two series theorem.

Observe that $\mathcal{M}_{f}^{(c)}(t)=g_{f}(T^f_t((\mathcal{U}_k^{(c)})_{k\geq 1}))$, $t\in\mathbb{R}$. We need to check that, for every $\phi\in L_1(\mathbb{C},\mathfrak{m})$,
\begin{equation}\label{eq:qrgodicity_proof1}
\lim_{T\to+\infty}\frac{1}{T}\int_0^T (\phi\circ g_{f})(T^f_t((\mathcal{U}_k^{(c)})_{k\geq 1})){\rm d}t=\E \phi(\mathcal{M}_{f}^{(c)}(0))\quad\mathfrak{H}-\text{a.s.}
\end{equation}
Assume we have already proved that, for every $h:\mathbb{T}_f^{(c)}\to\mathbb{C}$ satisfying $h\in L_1(\mathbb{T}_f^{(c)},\mathfrak{H})$, 
\begin{equation}\label{eq:qrgodicity_proof2}
\lim_{T\to+\infty}\frac{1}{T}\int_0^T h(T^f_t((\mathcal{U}_k^{(c)})_{k\geq 1})){\rm d}t=\int_{\mathbb{T}_f^{(c)}}h({\bf z})\mathfrak{H}({\rm d}{\bf z})\quad\mathfrak{H}-\text{a.s.}
\end{equation}
Then limit relation~\eqref{eq:qrgodicity_proof1} follows by choosing $h=\phi\circ g_{f}$ and noting that
$$
\int_{\mathbb{T}_f^{(c)}}\phi(g_{f}({\bf z}))\mathfrak{H}({\rm d}{\bf z})=\E \phi(\mathcal{M}_{f}^{(c)}(0)).
$$

Now we prove~\eqref{eq:qrgodicity_proof2}. First of all note that the $\mathfrak{H}-\text{a.s.}$ convergence to {\em some} limit holds by Birkhoff's theorem and we only need to verify that this limit is given by the right-hand side of~\eqref{eq:qrgodicity_proof2}. We shall do this in a sequence of steps. Observe that, for every fixed $T>0$, the mapping
$$
L^1(\mathbb{T}_f^{(c)},\mathfrak{H})\ni h\longmapsto \frac{1}{T}\int_0^T h(T^f_t(\cdot)){\rm d}t\in L^1(\mathbb{T}_f^{(c)},\mathfrak{H})=:\mathcal{A}^T_f(h),\quad T > 0,
$$
is a linear operator on $L^1(\mathbb{T}_f^{(c)},\mathfrak{H})$ with the operator norm bounded by $1$.

\vspace{2mm}
\noindent
{\sc Step 1.} We shall first verify that the set of measurable functions from $\mathbb{T}_f^{(c)}$ to $\mathbb{C}$, which depend only on finitely many coordinates, is dense in $L_1(\mathbb{T}_f^{(c)},\mathfrak{H})$ and hence the subset of continuous functions, which depend only on finitely many coordinates, is also dense in $L_1(\mathbb{T}_f^{(c)},\mathfrak{H})$. To see this, fix $h\in L_1(\mathbb{T}_f^{(c)},\mathfrak{H})$ and, for $n\in\mathbb{N}$, let $\mathcal{F}_n$ be the $\sigma$-algebra generated by $\mathcal{U}^{(c)}_{1},\ldots,\mathcal{U}^{(c)}_{n}$. The sequence of random variables
$$
\E_{\mathfrak{H}}[h((\mathcal{U}_k^{(c)})_{k\geq 1})|\mathcal{F}_n],\quad n\in\mathbb{N}
$$
is an $L_1(\mathbb{T}_f^{(c)},\mathfrak{H})$-bounded martingale and converges in $L_1(\mathbb{T}_f^{(c)},\mathfrak{H})$ to $h((\mathcal{U}_k^{(c)})_{k\geq 1})$, as $n\to\infty$, see Proposition 10.8.6 in~\cite{Resnick}. This can be recast as
$$
\lim_{n\to\infty}\int_{\mathbb{T}_f^{(c)}}|h_n(z_{\lambda_1^{(c)}},\ldots,z_{\lambda_n^{(c)}})-h({\bf z})|\mathfrak{H}({\rm d}{\bf z})=0,
$$
where
$$
h_n(u_1,\ldots,u_n):=\int_{\mathbb{T}_f^{(c)}}h(u_1,\ldots,u_n,z_{n+1},z_{n+2},\ldots)\mathfrak{H}({\rm d} {\bf z}),\quad n\in\N.
$$

\vspace{2mm}
\noindent
{\sc Step 2.} Assume we have proved~\eqref{eq:qrgodicity_proof2} on the dense subset from Step 1, that is,
\begin{equation}\label{eq:Kronecker-Weyl}
\lim_{T\to+\infty}\frac{1}{T}\int_0^T \widetilde{h}(\mathcal{U}_{1}^{(c)}\ee^{\ii \lambda_{1}^{(c)} t},\ldots,\mathcal{U}^{(c)}_{n}\ee^{\ii \lambda_n^{(c)} t}){\rm d}t=\int_{\mathbb{T}_f^{(c)}(n)}\widetilde{h}(z_{1},\ldots,z_{n}){\rm d}z_{1}\cdots{\rm d}z_n\quad\mathfrak{H}-\text{a.s.}
\end{equation}
for any continuous $\widetilde{h}:\mathbb{T}_f^{(c)}(n) \to \mathbb{C}$, where $\mathbb{T}_f^{(c)}(n):=\mathbb{T}_{\lambda^{(c)}_{1}}\times\cdots\times \mathbb{T}_{\lambda^{(c)}_{n}}$. In other words, that the sequence of linear operators $\mathcal{A}_f^T$ converges $\mathfrak{H}-\text{a.s.}$ pointwise on a dense subset of $L_1(\mathbb{T}_f^{(c)},\mathfrak{H})$ to the right-hand side of~\eqref{eq:qrgodicity_proof2}, as $T\to+\infty$. Since $h$ in our dense subset are continuous we infer, using the Lebesgue dominated convergence theorem, that $\mathcal{A}_f^T$ converges to the right-hand side of~\eqref{eq:qrgodicity_proof2} also in $L_1(\mathbb{T}_f^{(c)},\mathfrak{H})$ pointwise on a dense subset of $L_1(\mathbb{T}_f^{(c)},\mathfrak{H})$, as $T\to+\infty$. Since the operator norms of $\mathcal{A}_f^T$ are bounded by $1$, this convergence also holds on the whole space $L_1(\mathbb{T}_f^{(c)},\mathfrak{H})$. This argument identifies the limit in~\eqref{eq:qrgodicity_proof2}.

\vspace{2mm}
\noindent
{\sc Step 3. Proof of~\eqref{eq:Kronecker-Weyl}.} Formula~\eqref{eq:Kronecker-Weyl} follows by the continuous multidimensional Kronecker-Weyl equidistribution theorem, see, for example, Eq.~(1.21) in~\cite{Beck:2018}, with the right-hand side replaced by
$$
\int_{\mathbb{T}_f^{(c)}(n)} \widetilde{h}(\mathcal{U}_{1}^{(c)} z_{1},\ldots,\mathcal{U}_{n}^{(c)} z_n){\rm d}z_{1}\cdots{\rm d}z_n.
$$
This theorem is applicable because $\{\lambda_{1}^{(c)},\ldots,\lambda_n^{(c)}\}$ are assumed linearly independent over $\mathbb{Q}$. It remains to note that changing the variables in the  expression above gives
$$
\int_{\mathbb{T}_f^{(c)}(n)} \widetilde{h}(\mathcal{U}_{1}^{(c)} z_{1},\ldots,\mathcal{U}_{n}^{(c)} z_n){\rm d}z_{1}\cdots{\rm d}z_n=\int_{\mathbb{T}_f^{(c)}(n)} \widetilde{h}(z_{1},\ldots,z_n){\rm d}z_{1}\cdots{\rm d}z_n,\quad\mathfrak{H}-\text{a.s.}
$$
This proves~\eqref{eq:Kronecker-Weyl} and ergodicity of $\mathcal{M}_f^{(c)}$.

It remains to prove that $\mathcal{M}_{f}$ is not mixing. To this end, it suffices to check that the covariance function does not converge to $0$, as $s\to+\infty$. But this is clearly the case, since $K_{\mathcal{M}_{f}}$ is
non-zero Bohr's almost periodic and does not have a limit as $s\to+\infty$.
\end{proof}

The next result demonstrates that the choice of notation for the limiting random variable in~\eqref{eq:one_dim_convergence} was not coincident.

\begin{prop}\label{prop:convergence_at_fixed_point}
Assume that $f\in B_2$ satisfies~\eqref{eq:lin_independence_gen}. Then
$$
f(V_{-L,L})~\todistr~\mathcal{M}_{f}(0)=\sum_{k\geq 1}\left(\mathcal{U}_{k}^{(c)}f_{\lambda_k^{(c)}}+\overline{\mathcal{U}_{k}^{(c)}}{f_{-\lambda_k^{(c)}}}\right)+\sum_{k\in\Za}\mathcal{U}_{k}^{(r)}f_{\lambda_k^{(r)}},\quad L\to\infty,
$$
with the series on the right-hand side being a.s.~convergent.
\end{prop}
\begin{proof}
Let $f(t)~\sim~\sum_{\lambda\in\mathcal{S}(f)}f_{\lambda}\ee^{\ii \lambda t}$ be the formal Fourier series of $f$. The sequence of trigonometric polynomials
$$
p_n(t)=\sum_{k=1}^{n}f_{\lambda_k^{(c)}}\ee^{\ii \lambda_k^{(c)} t}+ \sum_{k=1}^{n}f_{-\lambda_k^{(c)}}\ee^{-\ii \lambda_k^{(c)} t} + \sum_{k=-n,k\neq 0}^{n}f_{\lambda_k^{(r)}}\ee^{\ii \lambda_k^{(r)} t},\quad n\in\mathbb{N},
$$
satisfies (by Theorem~14~(i) in~\cite{Duy:2013})
$$
\lim_{n\to\infty}\|f-p_n\|^2_{M_2}=\lim_{n\to\infty}\limsup_{L\to+\infty}\mathbb{E}\left|f(V_{-L,L})-p_n(V_{-L,L})\right|^2=0.
$$
For every fixed $n\in\mathbb{N}$, the set $\{\lambda_1^{(c)},\ldots,\lambda_n^{(c)},\lambda_1^{(r)},\ldots,\lambda_n^{(r)}\}$ is linearly independent over $\mathbb{Q}$. Thus, by the continuous multidimensional Kronecker-Weyl equidistribution theorem, see, for example, Eq.~(1.21) in~\cite{Beck:2018},
\begin{multline*}
p_n(V_{-L,L})=
\sum_{k=1}^{n}f_{\lambda_k^{(c)}}\ee^{\ii \lambda_k^{(c)} V_{-L,L}}+ \sum_{k=1}^{n}f_{-\lambda_k^{(c)}}\ee^{-\ii \lambda_k^{(c)} V_{-L,L}} + \sum_{k=-n,k\neq 0}^{n}f_{\lambda_k^{(r)}}\ee^{\ii \lambda_k^{(r)} V_{-L,L}}\\
~\todistr~\sum_{k=1}^{n}f_{\lambda_k^{(c)}}\mathcal{U}_k^{(c)}+ \sum_{k=1}^{n}f_{-\lambda_k^{(c)}}\overline{\mathcal{U}_k^{(c)}} + \sum_{k=-n,k\neq 0}^{n}f_{\lambda_k^{(r)}}\mathcal{U}_k^{(r)},\quad L\to+\infty.
\end{multline*}
It remains to note that, by Kolmogorov's two series theorem,
$$
\sum_{k=1}^{n}f_{\lambda_k^{(c)}}\mathcal{U}_k^{(c)}+ \sum_{k=1}^{n}f_{-\lambda_k^{(c)}} \overline{\mathcal{U}_k^{(c)}} + \sum_{k=-n,k\neq 0}^{n}f_{\lambda_k^{(r)}}\mathcal{U}_k^{(r)}~\toas~
\sum_{k\geq 1}f_{\lambda_k^{(c)}}\mathcal{U}_k^{(c)}+ \sum_{k\geq 1}f_{-\lambda_k^{(c)}}\overline{\mathcal{U}_k^{(c)}} +\sum_{k\in\Za}f_{\lambda_k^{(r)}}\mathcal{U}_k^{(r)},
$$
as $n\to\infty$, and apply Theorem 3.2 in~\cite{Billingsley}.
\end{proof}

The convergence of $f(V_{-L,L}+\cdot)$ on the space $B_2(\mathcal{S}(f))$, settled in Theorem~\ref{thm:conv_in_B2}, is a rather weak type of convergence of random processes, for it does not even  entail convergence in distribution at a fixed point. Proposition~\ref{prop:convergence_at_fixed_point} addresses this by establishing convergence in distribution at the fixed time $t=0$. Moreover, a simple argument can be used to extend this result to the convergence of any finite-dimensional distributions.
\begin{corollary}\label{cor:fidis_convergence}
Assume that $f\in B_2$ satisfies~\eqref{eq:lin_independence_gen}. Then
$$
(f(V_{-L,L}+t))_{t\in\mathbb{R}}~\Longrightarrow~(\mathcal{M}_{f}(t))_{t\in\mathbb{R}},\quad L\to+\infty
$$
in the sense of finite-dimensional distributions.
\end{corollary}
\begin{proof}
Fix $m\in\N$, $-\infty<t_1<\cdots<t_m<\infty$ and $\alpha_1,\ldots,\alpha_m\in\C$. According to the Cram\'{e}r-Wold device, it suffices to check that
\begin{multline}\label{eq:proof_fidis_C_R_conv}
\sum_{i=1}^{m}\alpha_i f(V_{-L,L}+t_i)~\todistr~\sum_{i=1}^{m}\alpha_i \mathcal{M}_{f}(t_i)\\
=\sum_{i=1}^{m}\alpha_i\left(\sum_{k\geq 1}\left(\mathcal{U}_k^{(c)}f_{\lambda_k^{(c)}}\ee^{\ii \lambda_k^{(c)} t_i}+\overline{\mathcal{U}_k^{(c)}}f_{-\lambda_k^{(c)}}\ee^{-\ii \lambda_k^{(c)} t_i}\right)+\sum_{k\in\Za}\mathcal{U}_k^{(r)}f_{\lambda_k^{(r)}}\ee^{\ii \lambda_k^{(r)} t_i}\right).
\end{multline}
The function $x\mapsto \sum_{i=1}^{m}\alpha_i f(x+t_i)=:\tilde{f}(x)$ belongs to the space $B_2(\mathcal{S}(f))$. According to Proposition~\ref{prop:convergence_at_fixed_point},
$$
\sum_{i=1}^{m}\alpha_i f(V_{-L,L}+t_i)~\todistr~\sum_{k\geq 1}(\mathcal{U}_k^{(c)}\tilde{f}_{\lambda_k^{(c)}}\ee^{\ii \lambda_k^{(c)} t_i}+\overline{\mathcal{U}_k^{(c)}}\tilde{f}_{-\lambda_k^{(c)}}\ee^{-\ii \lambda_k^{(c)} t_i})+\sum_{k\in\Za}\mathcal{U}_{k}^{(r)}\tilde{f}_{\lambda_k^{(r)}},\quad L\to\infty,
$$
where, for $k\geq 1$, $\widetilde{f}_{\pm\lambda_k^{(c)}}=f_{\pm\lambda_k^{(c)}}\sum_{i=1}^{m}\alpha_i \ee^{\pm\ii \lambda_k^{(c)} t_i}$ %, $k\geq 1$,
and $\widetilde{f}_{\lambda_k^{(r)}}=f_{\lambda_k^{(r)}}\sum_{i=1}^{m}\alpha_i \ee^{\pm\ii \lambda_k^{(r)} t_i}$ %, $k\geq 1$,
are the Fourier coefficients of $\tilde{f}$. This proves~\eqref{eq:proof_fidis_C_R_conv}.
\end{proof}

\subsection{Continuity of the limit process}

We shall now discuss continuity properties of the process $\mathcal{M}_f$. In general, the random process $\mathcal{M}_{f}$ might not be  continuous, even under assumption~\eqref{eq:lin_independence_gen}. A simple sufficient condition for continuity is
\begin{equation}\label{eq:fourier_coeffs_summable}
\sum_{\lambda\in\mathcal{S}(f)}|f_{\lambda}|<\infty,
\end{equation}
since the corresponding partial sums $\sum_{\lambda\in\mathcal{S}(f)\cap [-N,N]}f_{\lambda}\mathcal{V}_{\lambda} \ee^{\ii \lambda s}$ in this case converge to $\mathcal{M}_{f}(s)$ uniformly over $s\in\mathbb{R}$, as $N\to\infty$. Then $\mathcal{M}_{f}$ is also a Bohr almost periodic function with probability one.

If~\eqref{eq:fourier_coeffs_summable} does not hold we decompose $M_f$ into three sums
\begin{equation}\label{eq:continuity_decomposition}
\mathcal{M}_f(t)=\sum_{k\geq 1}\mathcal{U}_k^{(c)} f_{\lambda_k^{(c)}}\ee^{\ii \lambda_k^{(c)} t}+\sum_{k\geq 1}\overline{\mathcal{U}_k^{(c)}} f_{-\lambda_k^{(c)}}\ee^{\ii \lambda_k^{(c)} t}+\sum_{k\in\Za}\mathcal{U}_k^{(r)} f_{\lambda_k^{(r)}}\ee^{\ii \lambda_k^{(r)} t},\quad t\in\mathbb{R}
\end{equation}
and note that all of them are of the same type $\sum_{k}X_k\ee^{\ii \nu_k t}$ for a sequence $(X_k)$ of independent identically distributed random variables. For such random series, a sufficient condition of a.s.~sample path continuity can be found in Theorem 2 in~\cite{Hunt}. In the present case, this sufficient condition ensuring a.s.~continuity of $\mathcal{M}_f$ takes the form
\begin{equation*}
\sum_{k\geq 1}(|f_{\lambda_k^{(c)}}|^2+|f_{-\lambda_k^{(c)}}|^2)\log^{1+\varepsilon}_{+}|\lambda_k^{(c)}|+\sum_{k\in\Za}|f_{\lambda_k^{(r)}}|^2\log^{1+\varepsilon}_{+}|\lambda_k^{(r)}|<\infty\quad\text{for some}\quad \varepsilon>0
\end{equation*}
or, more compactly,
\begin{equation}\label{eq:hunt_continuous_suff}
\sum_{\lambda\in\mathcal{S}(f)}|f_{\lambda}|^2\log^{1+\varepsilon}_{+}|\lambda|<\infty\quad\text{for some}\quad \varepsilon>0.
\end{equation}
Assumptions~\eqref{eq:lambda_summatory_cond} and~\eqref{eq:a_summatory_cond} in the next result are variations of~\eqref{eq:hunt_continuous_suff} ensuring also H\"{o}lder %-
continuity of the paths.
\begin{prop}\label{thm:continuity}
Suppose that $f\in B_2$ satisfies~\eqref{eq:lin_independence_gen}. Assume that for some $\alpha\in [0,2)$ and $\beta>0$,
\begin{equation}\label{eq:lambda_summatory_cond}
\sum_{\lambda\in\mathcal{S}(f)}\lambda^2 |f_{\lambda}|^2\1_{\{|\lambda|\leq x\}}=O(x^{\alpha}),\quad x\to+\infty,
\end{equation}
and also
\begin{equation}\label{eq:a_summatory_cond}
\sum_{\lambda\in\mathcal{S}(f)}|f_{\lambda}|^2\1_{\{|\lambda|>x\}}=O(x^{-\beta}),\quad x\to+\infty.
\end{equation}
Then the process $\mathcal{M}_{f}$ has a version with a.s.~locally $\gamma$-H\"{o}lder %-
continuous paths for every $\gamma<\min(1-\alpha/2,\beta/2)$.
\end{prop}
\begin{proof}
The proof of Proposition~\ref{thm:continuity} is a standard application of the Kolmogorov-Chentsov theorem. First, observe that, by Rosenthal's inequality, see Theorem~3 in~\cite{Rosenthal_Subspaces:1970},
for some positive constant $C(m)>0$ which only depends on $m>1$,
\begin{align}
&\hspace{-1cm}\E |\mathcal{M}_{f}^{(r)}(t)-\mathcal{M}_{f}^{(r)}(0)|^{2m}=\E\left|\sum_{k\in\Za}\mathcal{U}_k^{(r)}f_{\lambda_k^{(r)}}(\ee^{\ii \lambda_{k}^{(r)}t}-1)\right|^{2m}\notag\\
&\leq C(m)\max\left\{ \left(\sum_{k\in\Za}|f_{\lambda_k^{(r)}}|^2|\ee^{\ii \lambda_k^{(r)}t}-1|^2\right)^m, \sum_{k\in\Za}|f_{\lambda_k^{(r)}}|^{2m}|\ee^{\ii \lambda_k^{(r)}t}-1|^{2m} \right\}\notag
\\
&= C(m) \left(\sum_{k\in\Za}|f_{\lambda_k^{(r)}}|^2|\ee^{\ii \lambda_k^{(r)}t}-1|^2\right)^m\label{eq:chentsov_proof_gen0},
\end{align}
where we have used that $(\mathcal{U}^{(r)}_{k})_{k\in\Za}$ are centered with % the
unit variance and unit $2m$-th absolute moment. We note in passing that the previous inequality holds as an equality in the case $m=1$, namely, $$\E |\mathcal{M}_{f}^{(r)}(t)-\mathcal{M}_{f}^{(r)}(0)|^2= \sum_{k\in\Za}|f_{\lambda_k^{(r)}}|^2|\ee^{\ii \lambda_k^{(r)}t}-1|^2.$$ Since all the three summands in~\eqref{eq:continuity_decomposition} have the same structure,~\eqref{eq:chentsov_proof_gen0} and the inequalities
$$
\left(\frac{a+b+c}{3}\right)^{2m}\leq\frac{a^{2m}+b^{2m}+c^{2m}}{3},\quad a,b,c\geq 0
$$
entail
\begin{equation}\label{eq:chentsov_proof_gen1}
\E |\mathcal{M}_{f}(t)-\mathcal{M}_{f}(0)|^{2m}\leq 3^{2m-1}C(m)\left(\sum_{\lambda\in\mathcal{S}(f)}|f_{\lambda}|^2|\ee^{\ii \lambda t}-1|^2\right)^m.
\end{equation}
Decompose the sum on the right-hand side as
\begin{equation}\label{eq:chentsov_proof_gen}
\sum_{\lambda\in\mathcal{S}(f)}|\ee^{\ii \lambda t}-1|^2|f_{\lambda}|^2\\
\leq 4t^2\sum_{\lambda\in\mathcal{S}(f)}\lambda^2|f_{\lambda}|^2 \1_{\{|\lambda|\leq t^{-1}\}}+4\sum_{\lambda\in\mathcal{S}(f)}|f_{\lambda}|^2\1_{\{|\lambda|>t^{-1}\}}.
\end{equation}
The first sum is $O(t^{-\alpha})$ in view of~\eqref{eq:lambda_summatory_cond}, whereas the second sum is $O(t^{\beta})$ in view of~\eqref{eq:a_summatory_cond}. Thus, the right-hand side of~\eqref{eq:chentsov_proof_gen} is $O(t^{\min(2-\alpha,\beta)})$, as $t\to 0+$. Plugging this into~\eqref{eq:chentsov_proof_gen1} implies that, for every $m\in\mathbb{N}$ and $T>0$, there exists $K=K(m,T)>0$ such that
$$
\E |\mathcal{M}_{f}(t)-\mathcal{M}_{f}(0)|^{2m}\leq K |t|^{m\min(2-\alpha,\beta)},\quad |t|\leq T.
$$
Picking $m\in\mathbb{N}$ such that $m\min(2-\alpha,\beta)>1$ completes the proof.
\end{proof}

%\begin{corollary}\label{corr:tightness}
%Under the assumptions of Proposition~\ref{thm:continuity}, the distributions of the continuous processes
%$$
%\left(\sum_{k=1}^{N}\mathcal{U}_k^{(c)} f_{\lambda_k^{(c)}}\ee^{\ii \lambda_k^{(c)} t}+\sum_{k=1}^{N}%\overline{\mathcal{U}_k^{(c)}} f_{-\lambda_k^{(c)}}\ee^{\ii \lambda_k^{(c)} t}+\sum_{k=-N,k\neq 0}^{N}\mathcal{U}_k^{(r)} f_{\lambda_k^{(r)}}\ee^{\ii \lambda_k^{(r)} t}\right)_{t\in\mathbb{R}},\quad N\in\N,
%$$
%are tight in the space of probability measures on $C(\R)$ endowed with the topology of locally uniform convergence.
%\end{corollary}
%\begin{proof}
%For every fixed $N\in\N$, the given process is stationary. Furthermore, it suffices to check tightness separately for each of the three sums. For the third summand, tightness follows from the Kolmogorov-Chentsov tightness criterion, since exactly the same calculations as we have used to derive~\eqref{eq:chentsov_gen}, give (take $m=2$)
%$$
%\E \left|\sum_{k=-N,k\neq 0}^{N}\mathcal{U}_k^{(r)} f_{\lambda_k^{(r)}} (\ee^{\ii \lambda_k^{(r)} t}-1)\right|^4\leq Kt^{\min(4-2\alpha,2\beta)},\quad t\in [0,t_0],
%$$
%for every fixed $t_0>0$ and all $N$ sufficiently large. The same reasoning applies to two other summands.
%\end{proof}

%\subsection{The case of linearly independent Fourier exponents: the tangent process}
\subsection{The tangent process}
In this section, we prove a functional limit theorem for a suitably normalized ``chordal process'' $(\mathcal{M}_f(\varepsilon t)-\mathcal{M}_f(0))_{t\in\mathbb{R}}$, as $\varepsilon\to 0+$. The limit process is called the tangent process of $\mathcal{M}_f$, a notion introduced by Falconer~\cite{Falconer_tangent,Falconer_local_structure}.

Observe that associated with any $f\in B_2\setminus \{0\}$ is a probability measure $\mathfrak{S}_f$ supported by the spectrum $\mathcal{S}(f)$ and defined by the equality
$$
\mathfrak{S}_f(\{\lambda\})=\frac{|f_{\lambda}|^2}{\sum_{\lambda\in\mathcal{S}(f)}|f_{\lambda}|^2}\overset{\eqref{eq:square_suumable}}{=}\frac{|f_{\lambda}|^2}{\|f\|^2_{M_2}},\quad \lambda\in\mathcal{S}(f).
$$
If $\xi_f$ is a random variable on $(\Omega,\mathcal{F},\mathbb{P})$ which is distributed according to $\mathfrak{S}_f$, our conditions~\eqref{eq:lambda_summatory_cond} and~\eqref{eq:a_summatory_cond} 
take the form
\begin{equation}\label{eq:a_summatory_cond_rv}
\mathbb{E}[\xi_f^2 \1_{\{|\xi_f|\leq x\}}]=O(x^{\alpha}),\quad x\to +\infty,
\end{equation}
and
\begin{equation}\label{eq:lambda_summatory_cond_rv}
\mathbb{P}[|\xi_f|\geq x]=O(x^{-\beta}),\quad x\to+\infty,
\end{equation}
respectively. In order to identify the tangent process of $\mathcal{M}_f$, we need more % precise
information about the behavior of Fourier exponents and Fourier coefficients of $f$ (encoded by the probability measure $\mathfrak{S}_f$).
More precisely, we shall assume that the truncated second moment of $\mathfrak{S}_f$ is a function which varies regularly at $+\infty$ with index $\theta\in (0,2)$, that is,
\begin{equation}\label{eq:xi_f_reg_var_moment}
\mathbb{E}[\xi_f^2 \1_{\{|\xi_f|\leq x\}}]~\sim~x^{\theta}\ell(x),\quad x\to+\infty.
\end{equation}
According to Eq.~(5.16) on p.~577 in~\cite{Feller_Vol2}, relation~\eqref{eq:xi_f_reg_var_moment} implies
\begin{equation}\label{eq:xi_f_reg_var_tail}
\mathbb{P}\{|\xi_f|>x\}~\sim~\frac{\theta}{2-\theta}x^{\theta-2}\ell(x),\quad x\to+\infty.
\end{equation}
Fix any $\varepsilon>0$. Since $\lim_{x\to +\infty} x^{-\varepsilon}\ell(x)=0$, we conclude that~\eqref{eq:xi_f_reg_var_moment} and~\eqref{eq:xi_f_reg_var_moment} secure~\eqref{eq:a_summatory_cond_rv} and~\eqref{eq:lambda_summatory_cond_rv} with $\alpha=\theta+\varepsilon$ and $\beta=2-\theta-\varepsilon$, respectively. By Proposition~\ref{thm:continuity}, this yields $\gamma$-H\"{o}lder continuity of $\mathcal{M}_f$ for each $\gamma<1-\theta/2$.
%We shall also make two further mild technical assumptions that
%\begin{equation}\label{eq:xi_f_reg_var_tail_balance}
%\begin{aligned}
%&\frac{\max\{\mathbb{P}\{\xi_f=\lambda\}\,:\,|\lambda| > x,\lambda\in\mathcal{S}(f)\}}{\mathbb{P}\{|\xi_f|>x\}}~\to~0,\quad x\to\infty,\\
%&\frac{\max\{\mathbb{P}\{\xi_f=\lambda\}\lambda^2\,:\,|\lambda| \leq x,\lambda\in\mathcal{S}(f)\}}{\mathbb{E}[\xi_f^2 \1_{\{|\xi_f|\leq x\}}]}~\to~0,\quad x\to\infty.
%\end{aligned}
%\end{equation}
%These assumptions simply mean that each individual atom in $\mathcal{S}(f)\cap[-x,x]^c$ (respectively, in $\mathcal{S}(f)\cap[-x,x]$) of the distribution of $\xi_f$ makes a negligible contribution to the tail (respectively, truncated second moment), as $x\to\infty$.

Theorem~\ref{thm:tangent} below is the promised functional limit theorem for the tangent process $(\mathcal{M}_f(\varepsilon t)-\mathcal{M}_f(0))_{t\in\mathbb{R}}$ under the additional assumption that the spectrum of $f$ is symmetric with respect to the origin, $\mathcal{R}=\varnothing$,  and the Fourier coefficients satisfy
\begin{equation}\label{eq:symmetric_coeffs}
f_{-\lambda}=\overline{f_{\lambda}},\quad \lambda\in\mathcal{S}(f).
\end{equation}
Assumption~\eqref{eq:symmetric_coeffs} implies that the distribution of $\xi_f$ is symmetric in the sense $\xi_f\overset{{\rm d}}{=}-\xi_f$. Under~\eqref{eq:symmetric_coeffs}, the process $\mathcal{M}_f=\mathcal{M}_f^{(c)}$ is given by
$$
\mathcal{M}_f(t)=\sum_{k\geq 1}\left(\mathcal{U}_k^{(c)} f_{\lambda_k^{(c)}}\ee^{\ii \lambda_k^{(c)} t}+\overline{\mathcal{U}_k^{(c)}}\overline{f_{\lambda_k^{(c)}}}\ee^{-\ii \lambda_k^{(c)} t}\right)=2\Re\left(\sum_{k\geq 1}\mathcal{U}_k^{(c)} f_{\lambda_k^{(c)}}\ee^{\ii \lambda_k^{(c)} t}\right).
$$
In particular, it is real-valued.

\begin{theorem}\label{thm:tangent}
Suppose that $f\in B_2$ satisfies~\eqref{eq:lin_independence_gen} with $\mathcal{R}=\varnothing$ and~\eqref{eq:symmetric_coeffs}. Assume that~\eqref{eq:xi_f_reg_var_moment} holds with some $\theta\in(0,2)$ and $\ell$ slowly varying at infinity.
%and also???? that~\eqref{eq:xi_f_reg_var_tail_balance} holds true.
Then,  as $\varepsilon\to 0+$,
\begin{equation}\label{eq:tangent_claim}
(T_{\varepsilon}(t))_{t\in\mathbb{R}}:=\left(\frac{\mathcal{M}_f(\varepsilon t)-\mathcal{M}_f(0)}{\sqrt{\varepsilon^{2-\theta}\ell(\varepsilon^{-1})}}\right)_{t\in\mathbb{R}}~\Longrightarrow~(B_{1-\theta/2}(t))_{t\in\mathbb{R}},
\end{equation}
on the space $C(\mathbb{R},\mathbb{R})$ with the locally uniform topology. Here $B_{1-\theta/2}$ is a centered fractional Brownian motion with the covariance
$$
\E B_{1-\theta/2}(t)B_{1-\theta/2}(s)=C_{f,\theta}(|t|^{2-\theta}+|s|^{2-\theta}-|t-s|^{2-\theta}),\quad t,s\in\mathbb{R},
$$
where $C_{\theta,f}$ is a positive constant given by equality~\eqref{eq:constant_def} below.
\end{theorem}
\begin{proof}[Proof of Theorem~\ref{thm:tangent}]
Observe that
$$
\mathcal{M}_f(\varepsilon t)-\mathcal{M}_f(0)=\sum_{k\geq 1}\left(\mathcal{U}_k^{(c)} f_{\lambda_k^{(c)}}(\ee^{\ii \lambda_k^{(c)} \varepsilon t}-1)+\overline{\mathcal{U}_k^{(c)}} \overline{f_{\lambda_k^{(c)}}}(\ee^{-\ii \lambda_k^{(c)} \varepsilon t}-1)\right)
$$ is the sum of independent random variables. Thus, a standard way to check the convergence is to apply the Lindeberg-Feller sufficient conditions. To this end, it suffices to check (i) convergence of covariances, (ii) the Lindeberg condition (we shall check the Lyapunov condition instead) and (iii) tightness.

\noindent
{\sc Convergence of covariances.} % Note that,
Using $\mathbb{E}\mathcal{U}_k^{(c)}=\mathbb{E}(\mathcal{U}_k^{(c)})^2=0$ and independence, we conclude that
\begin{align*}
\mathbb{E}[T_{\varepsilon}(t)T_{\varepsilon}(s)]&=\frac{1}{\varepsilon^{2-\theta}\ell(\varepsilon^{-1})}\sum_{k\geq 1}|f_{\lambda_k^{(c)}}|^2((\ee^{\ii \lambda_k^{(c)} \varepsilon t}-1)(\ee^{-\ii \lambda_k^{(c)} \varepsilon s}-1)+(\ee^{-\ii \lambda_k^{(c)} \varepsilon t}-1)(\ee^{\ii \lambda_k^{(c)} \varepsilon s}-1))\\
&=\frac{2}{\varepsilon^{2-\theta}\ell(\varepsilon^{-1})}\sum_{k\geq 1}|f_{\lambda_k^{(c)}}|^2 (\cos(\lambda_k^{(c)}\varepsilon (t-s))-\cos(\lambda_k^{(c)} \varepsilon t)-\cos(\lambda_k^{(c)} \varepsilon s)+1)\\
&=\frac{1}{\varepsilon^{2-\theta}\ell(\varepsilon^{-1})}\sum_{\lambda\in\mathcal{S}(f)}|f_{\lambda}|^2 (\cos(\lambda \varepsilon (t-s))-\cos(\lambda \varepsilon t)-\cos(\lambda \varepsilon s)+1)\\
&=\frac{\|f\|_{M_2}}{\varepsilon^{2-\theta}\ell(\varepsilon^{-1})}(v(\varepsilon t)+v(\varepsilon s)-v(\varepsilon (t-s)),
\end{align*}
where $v(t):=\mathbb{E}(1-\cos(\xi_f t))$, for $t\in\mathbb{R}$. In view of~\eqref{eq:xi_f_reg_var_tail}, Theorem 8.1.10 in~\cite{BGT} ensures that
\begin{equation}\label{eq:pitman_asymp}
\mathbb{E}[1-\cos(\xi_f \varepsilon)]~\sim~\frac{\pi\theta \varepsilon^{2-\theta}\ell(\varepsilon^{-1})}{2\Gamma(3-\theta)\sin(\pi\theta/2)},\quad \varepsilon\to 0+.
\end{equation}
Thus, for every fixed $t,s\in\mathbb{R}$,
\begin{equation}\label{eq:tangent_process_covariance}
\mathbb{E}[T_{\varepsilon}(t)T_{\varepsilon}(s)]~\to~ C_{f,\theta}(|t|^{2-\theta}+|s|^{2-\theta}-|t-s|^{2-\theta}),\quad \varepsilon\to 0+,
\end{equation}
where
\begin{equation}\label{eq:constant_def}
C_{f,\theta}:=\frac{\pi\theta\|f\|_{M_2}}{2\Gamma(3-\theta)\sin(\pi\theta/2)}.
\end{equation}
%The proof of part (i) is complete.

\noindent
{\sc Negligibility of atoms.} Before we can proceed with the verification of the Lyapunov condition, we show that each individual atom in $\mathcal{S}(f)\cap[-x,x]^c$ (respectively, in $\mathcal{S}(f)\cap[-x,x]$) of the distribution of $\xi_f$ makes a negligible contribution to the tail (respectively, truncated second moment), as $x\to\infty$.
\begin{lemma}
Assuming~\eqref{eq:xi_f_reg_var_moment} (hence, also~\eqref{eq:xi_f_reg_var_tail}) the following holds true
\begin{equation}\label{eq:xi_f_reg_var_tail_balance}
\begin{aligned}
&\frac{\max\{\mathbb{P}\{\xi_f=\lambda\}\,:\,|\lambda| > x,\lambda\in\mathcal{S}(f)\}}{\mathbb{P}\{|\xi_f|>x\}}~\to~0,\quad x\to+\infty,\\
&\frac{\max\{\mathbb{P}\{\xi_f=\lambda\}\lambda^2\,:\,|\lambda| \leq x,\lambda\in\mathcal{S}(f)\}}{\mathbb{E}[\xi_f^2 \1_{\{|\xi_f|\leq x\}}]}~\to~0,\quad x\to+\infty.
\end{aligned}
\end{equation}
\end{lemma}
\begin{proof}
To prove the first relation, let $R(x) := \mathbb{P}\{|\xi_f|>x\}$ for $x\geq 0$. By~\eqref{eq:xi_f_reg_var_tail},  $R(x)$ is regularly varying at $\infty$ with negative index $\theta-2$. Fix $\varepsilon>0$. By the uniform convergence theorem, see Theorem~1.5.2 in~\cite{BGT}, $0\leq 1 - \frac{R(y+1)}{R(y)}\leq \varepsilon$ for all sufficiently large $y$. It follows that, for sufficiently large $x$,
\begin{align*}
\frac{\max\{\mathbb{P}\{\xi_f=\lambda\}\,:\,|\lambda| > x,\lambda\in\mathcal{S}(f)\}}{\mathbb{P}\{|\xi_f|>x\}}
\leq \frac{\sup_{y\geq x} (R(y)- R(y+1))}{R(x)}\leq \frac{\varepsilon \sup_{y\geq x} % \varepsilon
R(y)}{R(x)}=\varepsilon.
\end{align*}
The proof of the second relation is similar. Define $T(x) := \mathbb{E}[\xi_f^2 \1_{\{|\xi_f|\leq x\}}]$ for $x\geq 0$. By~\eqref{eq:xi_f_reg_var_moment}, $T(x)$ is regularly varying at $\infty$ with positive index $\theta$. Fix $\varepsilon>0$. Again, by Theorem~1.5.2 in~\cite{BGT}, % we have
$0\leq \frac{T(y+1)}{T(y)} -1 \leq \varepsilon$ whenever $y\geq c_0$ for some sufficiently large $c_0$. Therefore, for sufficiently large $x$,
\begin{multline*}
\frac{\max\{\mathbb{P}\{\xi_f=\lambda\}\lambda^2\,:\,c_0 \leq |\lambda| \leq x,\lambda\in\mathcal{S}(f)\}}{\mathbb{E}[\xi_f^2 \1_{\{|\xi_f|\leq x\}}]}\\
\leq \frac{\sup_{y\in [c_0,x]} (T(y+1) - T(y))}{T(x)} \leq
\frac{\varepsilon\sup_{y\in [c_0,x]} T(y)}{T(x)} \leq \varepsilon,
\end{multline*}
which implies the second relation because % since
the maximal atom of $|\xi_f|$ in $[0,c_0]$ is bounded, whereas $T(x)\to\infty$.
\end{proof}

\noindent
{\sc The Lyapunov condition.} Fix $t\in\mathbb{R}$ and put
$\mathcal{Y}_{t,k}(\varepsilon):=2\Re(\mathcal{U}_k^{(c)}f_{\lambda_k^{(c)}}(\ee^{\ii \lambda_k^{(c)} \varepsilon t}-1))$. Then the Lindeberg condition reads as follows: for every fixed $\delta>0$ and $t\in\mathbb{R}$,
\begin{equation}\label{eq:lindeberg}
\frac{1}{\varepsilon^{2-\theta}\ell(\varepsilon^{-1})}\sum_{k\geq 1}\mathbb{E}[|\mathcal{Y}_{t,k}(\varepsilon)|^2 \1_{\{|\mathcal{Y}_{t,k}(\varepsilon)|^2>\delta \varepsilon^{2-\theta}\ell(\varepsilon^{-1})\}}]\to 0,\quad \varepsilon\to0+.
\end{equation}
We shall prove~\eqref{eq:lindeberg} by checking a stronger Lyapunov condition
\begin{equation}\label{eq:lyapunov}
\frac{1}{(\varepsilon^{2-\theta}\ell(\varepsilon^{-1}))^2}\sum_{k\geq 1}\mathbb{E}[|\mathcal{Y}_{t,k}(\varepsilon)|^4]\to 0,\quad \varepsilon\to0+.
\end{equation}
Using
$$
|\mathcal{Y}_{t,k}(\varepsilon)|^4\leq 16|f_{\lambda_k^{(c)}}|^4|\ee^{\ii \lambda_k^{(c)} \varepsilon t}-1|^4
$$
we conclude that
\begin{align*}
\sum_{k\geq 1}\mathbb{E}[|\mathcal{Y}_{t,k}(\varepsilon)|^4]&\leq 16\max_{k\geq 1}\left(|f_{\lambda_k^{(c)}}|^2|\ee^{\ii \lambda_k^{(c)} \varepsilon t}-1|^2\right)\sum_{k\geq 1}|f_{\lambda_k^{(c)}}|^2|\ee^{\ii \lambda_k^{(c)} \varepsilon t}-1|^2\\
&=8\|f\|^2_{M_2}\max_{k\geq 1}\left(|f_{\lambda_k^{(c)}}|^2|\ee^{\ii \lambda_k^{(c)} \varepsilon t}-1|^2\right) \mathbb{E}|\ee^{\ii \xi_f \varepsilon t}-1|^2\\
&\leq 16\|f\|^2_{M_2}\max_{\lambda\in\mathcal{S}(f)}\left(|f_{\lambda}|^2|\ee^{\ii \lambda \varepsilon t}-1|^2\right) \mathbb{E}(1-\cos(\xi_f \varepsilon t)).
\end{align*}
According to~\eqref{eq:pitman_asymp}, relation~\eqref{eq:lyapunov} follows once we can check that
\begin{equation*}
\frac{\max_{\lambda\in\mathcal{S}(f)}\left(|f_{\lambda}|^2|\ee^{\ii \lambda \varepsilon t}-1|^2\right)}{\varepsilon^{2-\theta}\ell(\varepsilon^{-1})}~\to~0,\quad \varepsilon\to0+.
\end{equation*}
Observe that
\begin{align*}
\max_{\lambda\in\mathcal{S}(f)}\left(|f_{\lambda}|^2|\ee^{\ii \lambda \varepsilon t}-1|^2\right)\leq\max_{\lambda\in\mathcal{S}(f)\,:\,\lambda>\varepsilon^{-1}}\left(|f_{\lambda}|^2|\ee^{\ii \lambda \varepsilon t}-1|^2\right)+\max_{\lambda\in\mathcal{S}(f)\,:\,\lambda\leq \varepsilon^{-1}}\left(|f_{\lambda}|^2|\ee^{\ii \lambda \varepsilon t}-1|^2\right)\\
\leq 4\max_{\lambda\in\mathcal{S}(f)\,:\,\lambda>\varepsilon^{-1}}|f_{\lambda}|^2+4t^2\varepsilon^2\max_{\lambda\in\mathcal{S}(f)\,:\,\lambda\leq \varepsilon^{-1}}\left(|f_{\lambda}|^2\lambda^2\right).
\end{align*}
Upon dividing by $\varepsilon^{2-\theta}\ell(\varepsilon^{-1})$ the first summand tends to zero by the first part of~\eqref{eq:xi_f_reg_var_tail_balance}. The second summand does so by the second part of~\eqref{eq:xi_f_reg_var_tail_balance}. Thus,~\eqref{eq:lindeberg} holds true.

\vspace{2mm}
\noindent
{\sc Tightness.} For the proof of tightness observe that,~\eqref{eq:chentsov_proof_gen1} and~\eqref{eq:chentsov_proof_gen} imply, for $t\neq s$ and $m\in\mathbb{N}$,
\begin{multline}\label{eq:tightness_tangent1}
\E [T_{\varepsilon}(t)-T_{\varepsilon}(s)]^{2m}=\frac{\E [\mathcal{M}_{f}(\varepsilon t)-\mathcal{M}_{f}(\varepsilon s)]^{2m}}{(\varepsilon^{2-\theta}\ell(\varepsilon^{-1}))^m}=\frac{\E [\mathcal{M}_{f}(\varepsilon (t-s))-\mathcal{M}_{f}(0)]^{2m}}{(\varepsilon^{2-\theta}\ell(\varepsilon^{-1}))^m}\\
\leq \frac{4^m 3^{2m-1}C(m)}{(\varepsilon^{2-\theta}\ell(\varepsilon^{-1}))^m}\left((\varepsilon |t-s|)^2\sum_{\lambda\in\mathcal{S}(f)}\lambda^2|f_{\lambda}|^2 \1_{\{|\lambda|\leq (\varepsilon|t-s|)^{-1}\}}+\sum_{\lambda\in\mathcal{S}(f)}|f_{\lambda}|^2\1_{\{|\lambda|>(\varepsilon |t-s|)^{-1}\}}\right)^m.
\end{multline}
Since $2-\theta>0$, the limit relations
$$
\frac{1}{\varepsilon^{2-\theta}\ell(\varepsilon^{-1})}(\varepsilon z)^2\sum_{\lambda\in\mathcal{S}(f)}\lambda^2|f_{\lambda}|^2 \1_{\{|\lambda|\leq (\varepsilon z)^{-1}\}}~\to~ \|f\|_{M_2}^2 z^{2-\theta},\quad \varepsilon\to 0+
$$
and
$$
\frac{1}{\varepsilon^{2-\theta}\ell(\varepsilon^{-1})}\sum_{\lambda\in\mathcal{S}(f)}|f_{\lambda}|^2 \1_{\{|\lambda|> (\varepsilon z)^{-1}\}}~\to~ \|f\|_{M_2}^2\frac{\theta}{2-\theta}z^{2-\theta},\quad \varepsilon\to 0+
$$
hold uniformly in $z\in (0,T]$ for every fixed $T>0$, % because $2-\theta>0$,
see Theorem 1.5.2 in~\cite{BGT}. Thus,~\eqref{eq:tightness_tangent1} implies that, for every $\varepsilon_0>0$, $T>0$ and $m\in\mathbb{N}$ there is a constant $C=C(\varepsilon_0,m,T)>0$ such that
$$
\E [T_{\varepsilon}(t)-T_{\varepsilon}(s)]^{2m}\leq C|t-s|^{(2-\theta)m},\quad \varepsilon\in (0,\varepsilon_0),\quad 0<|t-s|\leq T.
$$
Picking $m\in\mathbb{N}$ such that $(2-\theta)m>1$ we conclude that the family $(T_{\varepsilon}(t))_{t\in\mathbb{R}}$, $\varepsilon>0$ is tight in $C(\mathbb{R},\mathbb{C})$.
\end{proof}

\begin{example}\label{exmaple:tangent}
A general example of $f\in B_2$ satisfying all the assumptions of Theorem~\ref{thm:tangent} can be constructed as follows. Let $\ell_1,\ell_2$ be two functions which are slowly varying at infinity and $A>0$ % be
a positive constant. We take a countable set $\mathcal{C}$ of positive real numbers without finite accumulation points and such that $\mathcal{C}$ is linearly independent over $\mathbb{Q}$. We enumerate the points of $\mathcal{C}$ in the increasing order
$$
0<\lambda_1<\lambda_2<\cdots<\lambda_n<\cdots.
$$
Suppose that the sequence $(\lambda_k)_{k\in\mathbb{N}}$ is regularly varying at infinity with some index $\mathfrak{a}>0$, that is, $\lambda_n~\sim~n^{\mathfrak{a}}\ell_1(n)$, as $n\to\infty$. It is known that the function
$$
\lambda^{\leftarrow}(x):=\inf\{k\geq 1\,:\,\lambda_k > x\},\quad x > 0,
$$
is regularly varying at $+\infty$ with index $\mathfrak{a}^{-1}$, see Theorem 1.5.12 in~\cite{BGT}. Denote by $\mathcal{N}$ a random variable on $\mathbb{N}$ with distribution
\begin{equation}\label{eq:exmaple_N_reg_var}
\mathbb{P}\{\mathcal{N}=n\}~\sim~n^{-(\mathfrak{b}+1)}\ell_2(n),\quad n\to\infty
\end{equation}
for some $0<\mathfrak{b}<2\mathfrak{a}$. Let $f\in B_2$ be any function with the Fourier spectrum $\mathcal{S}(f)=\mathcal{C}\cup(-\mathcal{C})$ and the Fourier exponents $(f_{\lambda})_{\lambda\in\mathcal{S}(f)}$ satisfying\footnote{Note that the arguments of the complex numbers $f_{\lambda_k}$ are allowed to take any values in $\mathbb{T}$.}
$$
|f_{\lambda_k}|^2=A\mathbb{P}\{\mathcal{N}=k\}\quad\text{and}\quad f_{-\lambda_k}=\overline{f_{\lambda_k}},\quad k\geq 1
$$
for some fixed $A>0$. Observe that
$$
\mathbb{P}\{|\xi_f| > x\}=2\sum_{k\geq 1}\1_{\{\lambda_k > x\}}\mathbb{P}\{\mathcal{N}=k\}=2\mathbb{P}\{\mathcal{N}> \lambda^{\leftarrow}(x)\}.
$$
Since~\eqref{eq:exmaple_N_reg_var} implies that $x\mapsto \mathbb{P}\{\mathcal{N}>x\}$ is regularly varying at infinity with index $-\mathfrak{b}$, we conclude that $x\mapsto \mathbb{P}\{|\xi_f| > x\}$ satisfies~\eqref{eq:xi_f_reg_var_moment} and~\eqref{eq:xi_f_reg_var_tail} with $\theta=2-\frac{\mathfrak{b}}{\mathfrak{a}}$.
\end{example}

\subsection{Locally uniform convergence: a conjecture}

Corollary~\ref{cor:fidis_convergence} establishes that the finite-\-dim\-en\-sion\-al distributions of $(f(V_{-L,L}+t))_{t\in\mathbb{R}}$ converge to those of $(\mathcal{M}_f(t))_{t\in\mathbb{R}}$, while Theorem~\ref{thm:conv_in_B2} addresses the convergence in the space  $B_2(\mathcal{S}(f))$. This raises the question of whether a convergence also holds in a more classical functional space, such as the space of continuous functions $C(\mathbb{R},\mathbb{C})$ or the Skorokhod space $D(\mathbb{R},\mathbb{C})$. Naturally, for this question to be meaningful, the converging and limiting processes have to belong to the corresponding 
spaces with probability one.

Recall that either~\eqref{eq:hunt_continuous_suff} or \eqref{eq:lambda_summatory_cond} and~\eqref{eq:a_summatory_cond} ensure that $\mathcal{M}_{f}$ has a.s.~continuous sample paths. It also clear that $f(V_{-L,L}+\cdot)$ is a.s.~continuous if, and only if, so is $f$. We conjecture that, under these assumptions, the convergence of finite-dimensional distributions in Corollary~\ref{cor:fidis_convergence} can be lifted to the convergence in $C(\mathbb{R},\mathbb{C})$ endowed with the topology of locally uniform convergence.

\begin{hypothesis}\label{hyp:conv_in_C_R}
Assume that $f\in B_2\cap C(\R,\C)$ satisfies~\eqref{eq:lin_independence_gen} and either~\eqref{eq:hunt_continuous_suff}, or \eqref{eq:lambda_summatory_cond} and \eqref{eq:a_summatory_cond}, hold true. Then,
$$
(f(V_{-L,L}+t))_{t\in\mathbb{R}}~\Longrightarrow~(\mathcal{M}_{f}(t))_{t\in\mathbb{R}},\quad L\to+\infty,
$$
on $C(\mathbb{R},\C)$ endowed with the topology of locally uniform convergence.
\end{hypothesis}

Since our primary motivation stems from analytic number theory, and the almost periodic functions that arise in our applications are not continuous, we do not pursue a proof of this conjecture. Instead, we only
note thatTheorems 1 and 2 in~\cite{Folner:1945} may be useful for checking tightness, which is the only part that requires a proof. The aforementioned theorems justify a uniform approximation of $f$ by trigonometric polynomials 
on a subset of $\mathbb{R}$ of upper density arbitrarily close to one.

\section{Applications in analytic number theory}\label{sec:examples}

Almost periodic functions play a significant role in number theory, particularly in the analysis
of asymptotic distributions of arithmetic functions. In what follows we denote by $\mathcal{P}$ the set of prime numbers and by $\zeta$ the Riemann zeta-function defined by the series $\zeta(s):=\sum_{n\geq 1}\frac{1}{n^s}$ for $\Re(s)>1$ and by an analytic continuation for all other $s\in\C$, $s\neq 1$.

\subsection{The von Mangoldt function and its partial sums}\label{subsec:mangoldt}
To the best of our knowledge, the first connection between almost periodic functions and analytic number theory is due to Aurel Wintner~\cite{Wintner:1941} who considered partial sums of the von Mangoldt function. Recall that the von Mangoldt function is a mapping $\Lambda:\N \to \R$ defined by
$$
\Lambda(n)=
\begin{cases}
\log p,&\text{if }n=p^j\text{ for some }p\in\mathcal{P}\text{ and }j\in\N,\\
0,& \text{otherwise}.
\end{cases}
$$
Its partial sum (or summatory) function defined by $\psi(x)=\sum_{n\leq x}\Lambda(n)$, $x\geq 0$, is called the second Chebyshev function. It is known that the prime number theorem is equivalent to the asymptotic relation
\begin{equation}\label{eq:pnt_equiv_chebyshev}
\psi(x)~\sim~x,\quad x\to+\infty.
\end{equation}
A link between functions $\Lambda$ and $\zeta$ is settled by the formula
$$
\frac{\zeta'(s)}{\zeta(s)}=-\sum_{n\geq 1}\frac{\Lambda(n)}{n^s},\quad \Re(s)>1,
$$
which follows by taking the logarithmic derivative on both sides of Euler's product $$
1/\zeta(s)=\prod_{p\in\mathcal{P}}(1-p^{-s}).
$$
Equivalently, one can write
\begin{equation}\label{eq:zeta_recipro-chebyshev}
\frac{\zeta'(s)}{\zeta(s)}=-s\int_1^{\infty}\frac{\psi(x)}{x^{s+1}}\d x,\quad \Re(s)>1.
\end{equation}
Put $\widetilde{\psi}(x):=(\psi(x+0)+\psi(x-0))/2$. By Perron's inversion formula, see Theorem 11.18 in~\cite{Apostol}, we infer from~\eqref{eq:zeta_recipro-chebyshev}
\begin{equation}\label{eq:zeta_recipro-chebyshev_inverse}
\widetilde{\psi}(x)=-\frac{1}{2\pi\ii}\int_{c-\ii\infty}^{c+\ii\infty}\frac{\zeta'(s)x^s}{\zeta(s)s}\d s,\quad x>1,
\end{equation}
where $c>1$ is an arbitrary fixed constant.

Inversion formula~\eqref{eq:zeta_recipro-chebyshev_inverse} suggests that the asymptotic behavior of $\widetilde{\psi}$ is regulated by the poles of the integrands which are precisely the zeros of $\zeta$. Throughout this section we assume the Riemann hypothesis (RH) saying that all zeros of $\zeta$ in the critical strip $0<\Re(s)<1$ lie on the line $\Re(s)=1/2$, thus take the form $\rho=1/2+\ii \gamma$, $\gamma\in\R$. Assuming the RH we shall use the following notation for zeros of $\zeta$ and their imaginary parts. Put
$$
\mathcal{I}_{\zeta}:=\{\gamma>0:\zeta(1/2 + \ii \gamma)=0\},\quad \mathcal{Z}^{\pm}_{\zeta}:=1/2\pm\ii\cdot \mathcal{I}_{\zeta}.
$$
Thus, under the RH, the set of all zeros of $\zeta$ in the critical strip is $\mathcal{Z}_{\zeta}=\mathcal{Z}^{+}_{\zeta}\cup \mathcal{Z}^{-}_{\zeta}$. For $\rho\in \mathcal{Z}_{\zeta}$, we denote by $\gamma=\gamma(\rho)$ the imaginary part of $\rho$. The points in $\mathcal{Z}_{\zeta}$ are enumerated by non-zero integers $\Z^{\ast}:=\Z\setminus\{0\}$ such that $k<m$ implies $\gamma_k<\gamma_m$ and $\gamma_{-1}<0<\gamma_1$. Thus, a sum of the form $\sum_{\rho\in\mathcal{Z}_{\zeta}}a_{\rho}$ is understood as $\sum_{k\in\Za}a_{\rho_k}$.  Moreover, if the series does not converge absolutely we understand the latter as the limit
$\lim_{T\to +\infty}\sum_{k\in\Za:|\gamma_k|\leq T}a_{\rho_k}$.

Having introduced the above notation we can now state the von Mangoldt formula which is essentially a residue expansion in~\eqref{eq:zeta_recipro-chebyshev_inverse}, see Theorem 29 in~\cite{Ingham} for the formal derivation,
\begin{equation}\label{eq:von_mangoldt_formula}
\widetilde{\psi}(x)=x-\sum_{k\in\Za} \frac{x^{\rho_k}}{\rho_k}-\frac{\zeta'(0)}{\zeta(0)}-\frac{1}{2}\log\left(1-\frac{1}{x^2}\right),\quad x>1.
\end{equation}
Replacing here $x$ with $\ee^t$ and rearranging, we obtain
\begin{equation}\label{eq:von_mangoldt_formula_normalized}
\widehat{\psi}(t):=\frac{\widetilde{\psi}(\ee^t)-\ee^t}{\ee^{t/2}}=-\sum_{k\in\Za} \frac{\ee^{\ii\gamma_k t}}{\rho_k}-\frac{\zeta'(0)}{\zeta(0)}\ee^{-t/2}-\frac{\ee^{-t/2}}{2}\log\left(1-\ee^{-2t}\right)=-\psi_{ap}(t)+R_{\psi}(t),\quad t>0,
\end{equation}
where $\psi_{ap}(t)=\sum_{k\in\Za} \frac{\ee^{\ii\gamma_k t}}{\rho_k}$.

As the difference of a continuous and a discontinuous functions, $\psi_{ap}$ is not continuous. It is also not c\'{a}dl\'{a}g for, by definition, its value at a discontinuity point is half the
sum of left and right limits at this point. In particular, $\psi_{ap}$ cannot be a Bohr almost periodic function. On the other hand, we claim that $\psi_{ap}$ is a Besicovitch almost periodic function. Indeed, it is known, see pp.~97-100 in~\cite{Davenport}, that
\begin{equation}\label{eq:counting_zeros}
\sum_{\rho\in\mathcal{Z}_{\zeta}^{+}}\1_{\{\gamma\leq x\}}=\sum_{k\geq 1}\1_{\{\gamma_k\leq x\}}~\sim~\frac{x\log x}{2\pi},\quad x\to+\infty.
\end{equation}
By a standard Tauberian argument this implies that, for $s>1/2$,
\begin{equation}\label{eq:counting_inverse_zeros}
\sum_{\rho\in\mathcal{Z}_{\zeta}^{+}}\frac{\1_{\{\gamma>x\}}}{|\rho|^{2s}}=\sum_{k\geq 1}\frac{\1_{\{\gamma_k>x\}}}{|\rho_k|^{2s}}~\sim~\frac{x^{1-2s}\log x}{2\pi(2s-1)},\quad x\to+\infty,
\end{equation}
and  that the series $\sum_{\rho\in\mathcal{Z}_{\zeta}^{+}}|\rho|^{-2s}$ converges if, and only if, $s>1/2$. Using this fact, with $s=1$, demonstrates that $\psi_{ap}$ is a Besicovitch almost periodic function with the Fourier spectrum $\mathcal{Z}_{\zeta}$ and the Fourier coefficients $\{1/\rho\,:\,\rho\in\mathcal{Z}_{\zeta}\}$ by the Riesz--Fischer theorem, see p.~110 in~\cite{Besicovitch}. Observing that $R_{\psi}(t)\to 0$, as $t\to+\infty$, Wintner inferred using a one-sided version of formula~\eqref{eq:one_dim_convergence}, that $\widehat{\psi}(V_{0,L})$ converges in distribution, as $L\to+\infty$, to a nondegenerate random variable. Note that for $t=V_{0,L}$, formula~\eqref{eq:von_mangoldt_formula_normalized} holds with probability one when $\widetilde{\psi}$ is replaced by $\psi$.

The results of Section~\ref{sec:ap_processes} allow us to analyze {\em the process} $\widehat{\psi}(V_{0,L}+\cdot)$. The remainder term $R_{\psi}$ is an elementary function and its analysis is trivial. For example, it is clear that
$$
\sup_{t\in [a,b]}|R_{\psi}(V_{0,L}+t)|~\overset{\mathbb{P}}{\to}~0,\quad L\to+\infty,
$$
for every fixed $0<a<b$. For the principal term $\psi_{ap}$  we have the following generalization of Wintner's result.
\begin{prop}\label{prop:mangoldt_functional_B2}
Assume (RH). As $L\to+\infty$, the random processes $\psi_{ap}(V_{0,L}+\cdot)$ converge in distribution on the space $B_2(\mathcal{Z}_{\zeta})$ (and, therefore, on the space $B_2$) to the stationary random process $\mathcal{M}_{\psi_{ap}}$.
\end{prop}

The following hypothesis
\begin{equation}\label{eq:assumption_li}
\text{\em{The set }}\mathcal{I}_{\zeta}\text{ \em{is linearly independent over }}\Q,
\end{equation}
is widely believed to be true but its proof seems to be currently out of reach. Under this hypothesis,~\eqref{eq:lin_independence_gen} holds true with $\mathcal{C}=\mathcal{Z}_{\zeta}^{+}$ and $\mathcal{R}=\varnothing$. Moreover, the Fourier coefficients satisfy $\rho_{-k}=\overline{\rho_k}$ for $k\geq 1$. The process $\mathcal{M}_{\psi_{ap}}$ turns out to be a.s.~continuous, see Figure~\ref{fig:ChebMang} for a sample path realisation and Figure~\ref{fig:ChebMang_corr} for the correlation function, which is a Bohr almost periodic function.

\begin{prop}
Assume (RH) and~\eqref{eq:assumption_li}. Then, for all $t\in\mathbb{R}$, the limit process in Proposition~\ref{prop:mangoldt_functional_B2} has the Fourier expansion
$$
\mathcal{M}_{\psi_{ap}}(t)=\sum_{k\geq 1}(\rho_k^{-1}\mathcal{U}_k\ee^{\ii \gamma_k t}+\overline{\rho_k}^{-1}\overline{\mathcal{U}_k}\ee^{-\ii \gamma_k t})=2\Re\left(\sum_{k\geq 1}\rho_k^{-1}\mathcal{U}_k\ee^{\ii \gamma_k t}\right),
$$
which converges a.s. almost everywhere with $(\mathcal{U}_{k})_{k\geq 1}$ being independent random variables with the uniform distribution on the unit circle. The process $\mathcal{M}_{\psi_{ap}}$ has $\delta$-H\"{o}lder-continuous paths for every $\delta<1/2$. Moreover, $(\psi_{ap}(V_{0,L}+t))_{t\in\mathbb{R}}$ converge, as $L\to\infty$, in the sense of finite-dimensional distributions to $(\mathcal{M}_{\psi_{ap}}(t))_{t\in\mathbb{R}}$.
\end{prop}
\begin{proof}
The claim follows from Proposition~\ref{thm:continuity}. Indeed, in view of~\eqref{eq:counting_zeros}, formula~\eqref{eq:lambda_summatory_cond} holds with any $\alpha>1$, whereas~\eqref{eq:counting_inverse_zeros} implies that~\eqref{eq:a_summatory_cond} is true for any $\beta<1$.
\end{proof}
\begin{rem}
Note that condition~\eqref{eq:hunt_continuous_suff} is also satisfied, thereby ensuring a.s.~continuity of $\mathcal{M}_{\psi_{ap}}$ without the H\"{o}lder property.
\end{rem}
According to Theorem 25c in~\cite{Ingham}, asymptotic relation~\eqref{eq:counting_zeros} can be inverted to get
\begin{equation}\label{eq:ingham_gamma_growth}
\gamma_n~\sim~2\pi n/\log n,\quad n\to+\infty.
\end{equation}
Therefore, we are in the setting of Example~\ref{exmaple:tangent} with $\mathcal{C}=\mathcal{I}_{\zeta}$, $\mathfrak{a}=\mathfrak{b}=1$, which means that Theorem~\ref{thm:tangent} holds true for the process $\mathcal{M}_{\psi_{ap}}$ and the limit process $B_{1-\theta/2}=B_{1/2}$ is a standard Brownian motion.

\begin{figure}
\centering
\includegraphics[width=0.7\textwidth]{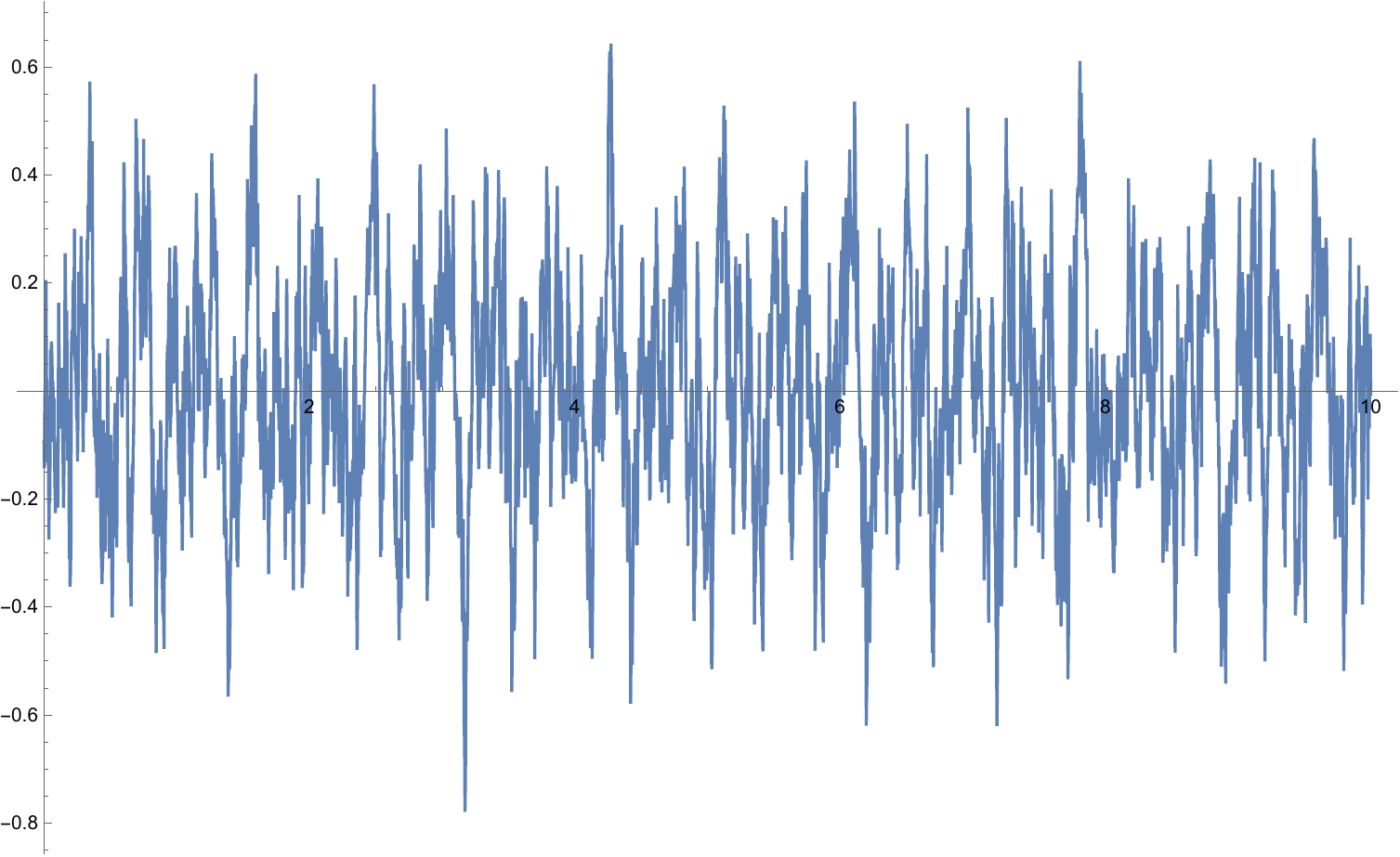}
\caption{Realisation of the process $\mathcal{M}_{\psi_{ap}}(t)=2\Re(\sum_{k\geq 1}\rho_k^{-1}\mathcal{U}_k\ee^{\ii \gamma_k t})$, $t\in [0,10]$, appearing in Section~\ref{subsec:mangoldt}.}
%Top: $t\in [0,10]$. Bottom: $t=0,1,\ldots, 10^5$.
    \label{fig:ChebMang}
\end{figure}

\begin{figure}
\centering
\includegraphics[width=0.7\textwidth]{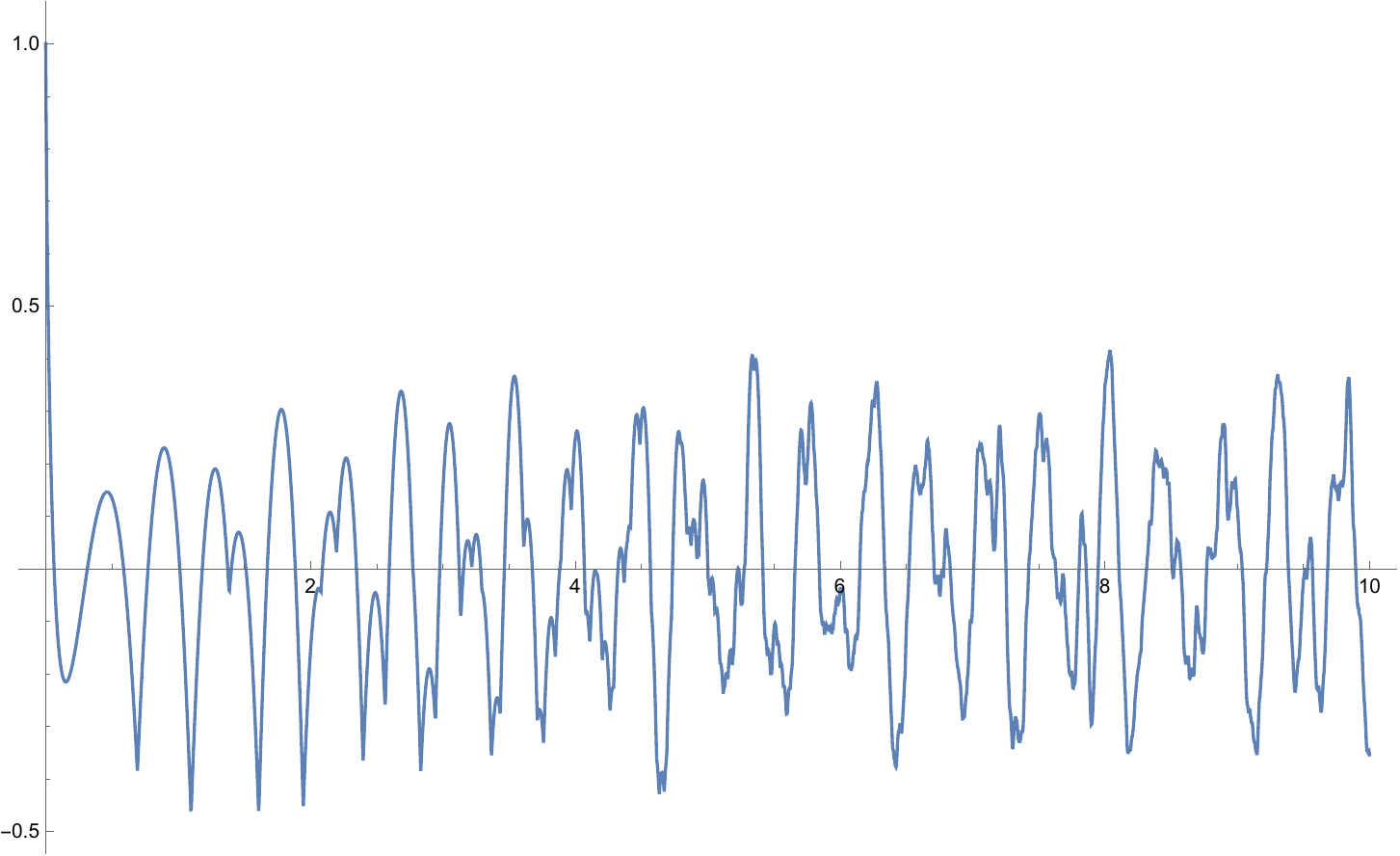}
\caption{Correlation function of the process $\mathcal{M}_{\psi_{ap}}$ appearing in Section~\ref{subsec:mangoldt}, for $t\in [0,10]$.}
    \label{fig:ChebMang_corr}
\end{figure}

\subsection{The M\"{o}bius and Mertens functions}\label{sec:mobius}
An analogue of Wintner's analysis of the Chebyshev function remainder has been carried out for the M\"{o}bius summatory function in~\cite{Ng:2004}. Recall that the M\"{o}bius function $\mu:\N\to\{-1,0,1\}$ is defined by
$$
\mu(n)=
\begin{cases}
1,&\text{if }n=1,\\
0,&\text{if }n\text{ is not square-free},\\
(-1)^k,&\text{if }n=p_1 p_2\cdots p_k\text{ for pair-wise distinct primes }p_1,\ldots,p_k. %\in\mathcal{P}.
\\
\end{cases}
$$
The summatory function $M(x)=\sum_{n\leq x}\mu(n)$ is called the Mertens function. A counterpart of~\eqref{eq:zeta_recipro-chebyshev_inverse} can be derived by noting that
\begin{equation}\label{eq:zeta_recipro-mertens}
\frac{1}{\zeta(s)}=\prod_{p\in\mathcal{P}}(1-p^{-s})=\sum_{n=1}^{\infty}\frac{\mu(n)}{n^s}=s\int_1^{\infty}\frac{M(x)}{x^{s+1}}\d x,\quad \Re(s)>1,
\end{equation}
which by Perron's inversion yields
\begin{equation}\label{eq:zeta_recipro-mertens_inverse}
\widetilde{M}(x)=\frac{1}{2\pi\ii}\int_{c-\ii\infty}^{c+\ii\infty}\frac{x^s}{s\zeta(s)}\d s,\quad x>1,
\end{equation}
where $c>1$ is an arbitrary fixed constant. Here, % as before,
$\widetilde{M}(x)=(M(x+0)+M(x-0))/2$. Assuming (RH) and also that all non-trivial zeros of $\zeta$ are simple, a residue expansion of the right-hand side in~\eqref{eq:zeta_recipro-mertens_inverse} yields
\begin{equation}\label{eq:mertens_via_zeros}
\frac{\widetilde{M}(x)}{\sqrt{x}}=\sum_{k\in\Za}\frac{x^{\ii\gamma_k}}{\rho_k\zeta'(\rho_k)}-\frac{2}{\sqrt{x}}+\sum_{n\geq 1}\frac{(-1)^{n-1}(2\pi)^{2n}x^{-2n-1/2}}{(2n)!n\zeta(2n+1)},\quad x>1,
\end{equation}
see Theorem 14.27 in~\cite{Titchmarsh:1986}. The series representing the first term on the right-hand side of~\eqref{eq:mertens_via_zeros} does not converge absolutely and is understood as the limit of the partial sums $\sum_{k\in\Za:|\gamma_k|\leq T}\frac{x^{\ii\gamma_k}}{\rho_k\zeta'(\rho_k)}$ along an appropriate sequence $T=T_n$ diverging to $+\infty$. Replacing $x$ with $\ee^{t}$, we obtain
$$
\widehat{M}(t):=\ee^{-t/2}\widetilde{M}(\ee^t)=:M_{ap}(t)+R_{M}(t),\quad t>0,
$$
where $M_{ap}(t)=\sum_{k\in\Za}\frac{\ee^{\ii \gamma_k t}}{\rho_k\zeta'(\rho_k)}$ is an almost periodic part and $R_{M}$ is the remainder.

In contrast with the case of the Chebyshev function, the condition of square summability of the Fourier coefficients which would guarantee that $M_{ap}\in B_2$, namely,
\begin{equation}\label{eq:mertens_square_summability}
\sum_{k\in\Za}\frac{1}{|\rho_k\zeta'(\rho_k)|^2}=2\sum_{k\geq 1}\frac{1}{|\rho_k\zeta'(\rho_k)|^2}<\infty,
\end{equation}
does not hold a priori. Hence, it has to be postulated for our purposes. A known sufficient condition for~\eqref{eq:mertens_square_summability} is the weak Mertens hypothesis
\begin{equation}\label{eq:mertens_square_summability_suff1}
\int_{0}^{T}\left(\frac{M(x)}{x}\right)^2{\rm d}x=O(\log T),\quad T\to+\infty,
\end{equation}
see Theorem 14.29(B) in~\cite{Titchmarsh:1986}. Another sufficient condition which will also play an important role in what follows is the weak Gonek conjecture
\begin{equation}\label{eq:gonek}
J_{-1}(T)=\sum_{k\in\Za:|\gamma_k|\leq T}\frac{1}{|\zeta'(\rho_k)|^2}=O(T),\quad T\to +\infty,
\end{equation}
see~\cite{Gonek} and also the discussion in the introduction to~\cite{Ng:2004}. It is clear that the weak Gonek
conjecture implicitly assumes that all zeros on the critical line are simple. Furthermore, according to Lemma 1 in~\cite{Ng:2004}, it entails~\eqref{eq:mertens_square_summability_suff1} and hence~\eqref{eq:mertens_square_summability}. Assuming (RH) and~\eqref{eq:gonek}, Theorem 2 in~\cite{Ng:2004} states that $\widehat{M}(V_{0,L})$ converges in distribution\footnote{Note that the remainder $R_M$ clearly satisfies $R_M(t)\to 0$, as $t\to+\infty$.}. Here is a functional version of this result in the space $B_2$, which follows immediately from Corollary~\ref{cor:one_sided_convergence_in_B2}.

\begin{prop}\label{prop:mertens_functional_B2}
Assume (RH) and~\eqref{eq:mertens_square_summability}. As $L\to+\infty$, the random processes $M_{ap}(V_{0,L}+\cdot)$ converge in distribution on the space $B_2(\mathcal{Z}_{\zeta})$ (and, therefore, on the space $B_2$) to the stationary random process $\mathcal{M}_{M_{ap}}$. If, in addition,~\eqref{eq:assumption_li} holds true, then $(M_{ap}(V_{0,L}+t))_{t\in\mathbb{R}}$ converge to
$(\mathcal{M}_{M_{ap}}(t))_{t\in\mathbb{R}}$ also in the sense of finite-dimensional distributions.
\end{prop}

Under the assumptions of Proposition \ref{prop:mertens_functional_B2}, including~\eqref{eq:assumption_li}, 
the random Fourier series
$$
\mathcal{M}_{M_{ap}}(t)=\sum_{k\geq 1}\left((\rho_k\zeta'(\rho_k))^{-1}\mathcal{U}_k\ee^{\ii \gamma_k t}+(\overline{\rho_k\zeta'(\rho_k)})^{-1}\overline{\mathcal{U}_k}\ee^{-\ii \gamma_k t}\right)=2\Re\left(\sum_{k\geq 1}(\rho_k\zeta'(\rho_k))^{-1}\mathcal{U}_k\ee^{\ii \gamma_k t}\right)
$$
converges a.s.~almost everywhere. Furthermore, in this case, Proposition~\ref{prop:mertens_functional_B2} recovers Theorem 2 in~\cite{Ng:2004} because
$$
\sup_{t\in [a,b]}|R_M(V_{0,L}+t)|\overset{\mathbb{P}}{\to} 0,\quad L\to+\infty %,
$$
for all $0<a<b$. If a slightly stronger version of~\eqref{eq:mertens_square_summability} holds true, namely,
\begin{equation}\label{eq:eq:mertens_square_summability_with_log}
\sum_{k\geq 1}\frac{\log_{+}^{1+\varepsilon}\gamma_k}{|\rho_k\zeta'(\rho_k)|^2}<\infty %,
\end{equation}
for some $\varepsilon>0$, then $\mathcal{M}_{M_{ap}}$ is a.s.~continuous in accordance with~\eqref{eq:hunt_continuous_suff}. Assuming~\eqref{eq:gonek}, which is stronger than~\eqref{eq:eq:mertens_square_summability_with_log} as follows from the proof of formula
~\eqref{eq:below} given below, we deduce H\"{o}lder continuity of $\mathcal{M}_{M_{ap}}$.

\begin{prop}\label{prop:mertens_continuity_limit}
Assume (RH),~\eqref{eq:assumption_li} and~\eqref{eq:gonek}. The process $\mathcal{M}_{M_{ap}}$ is a.s.~$\delta$-H\"{o}lder-continuous with $\delta<1/2$.
\end{prop}
\begin{proof}We apply Proposition~\ref{thm:continuity}. The left-hand side of~\eqref{eq:lambda_summatory_cond} reads
\begin{multline*}
\sum_{k\in\Za}\frac{\gamma_k^2}{|\rho_k\zeta'(\rho_k)|^2}\1_{\{|\gamma_k|\leq x\}}=2\sum_{k\geq 1}\frac{\gamma_k^2}{|\rho_k\zeta'(\rho_k)|^2}\1_{\gamma_k \leq x\}}\\
\leq 2\sum_{k\geq 1}\frac{1}{|\zeta'(\rho_k)|^2}\1_{\{\gamma_k
\leq x\}}=J_{-1}(x)=O(x),\quad x\to+\infty,
\end{multline*}
where the last $O$-estimate is just the weak Gonek conjecture~\eqref{eq:gonek}. The left-hand side of~\eqref{eq:a_summatory_cond} in the present situation is
$$
\sum_{k\in\Za}\frac{1}{|\rho_k\zeta'(\rho_k)|^2}\1_{\{|\gamma_k|>x\}}=2\sum_{k\geq 1}\frac{1}{|\rho_k\zeta'(\rho_k)|^2}\1_{\{\gamma_k>x\}}=\sum_{j\geq 1}\sum_{k\geq 1}\frac{1}{|\rho_k\zeta'(\rho_k)|^2}\1_{\{2^{j-1} x<\gamma_k \leq 2^{j}x\}}.
$$
By Lemma 1(ii) in~\cite{Ng:2004},~\eqref{eq:gonek} ensures that there exists a constant $K_1>0$ such that
\begin{equation}\label{eq:gonek2}
\sum_{k\geq 1}\frac{1}{|\rho_k\zeta'(\rho_k)|^2}\1_{\{2^{j-1} x< \gamma_k \leq 2^{j}x\}}\leq \frac{K_1}{2^{j-1}x} 
\end{equation}
for all $x$ large enough and $j\geq 1$. Summing over $j\geq 1$ yields
$$
\sum_{k\geq 1}\frac{1}{|\rho_k\zeta'(\rho_k)|^2}\1_{\{\gamma_k >x\}}=O(1/x),\quad x\to\infty.
$$
In order to see that~\eqref{eq:gonek2} (hence~\eqref{eq:gonek}) entails~\eqref{eq:eq:mertens_square_summability_with_log} observe that 
\begin{multline}\label{eq:below}
\sum_{k\geq 1}\frac{\log_{+}^{1+\varepsilon}\gamma_k}{|\rho_k\zeta'(\rho_k)|^2}=\sum_{j\geq 1}\sum_{k\geq 1}\frac{\log_{+}^{1+\varepsilon}\gamma_k}{|\rho_k\zeta'(\rho_k)|^2}\1_{\{2^{j-1}< \gamma_k\leq 2^{j}\}}\\
\leq \sum_{j\geq 1} \log_{+}^{1+\varepsilon}2^j \sum_{k\geq 1}\frac{1}{|\rho_k\zeta'(\rho_k)|^2}\1_{\{2^{j-1}< \gamma_k\leq 2^{j}\}}\leq 2K_1\sum_{j\geq 1}\frac{\log_{+}^{1+\varepsilon}2^j}{2^j}<\infty.
\end{multline}
\end{proof}

\subsection{The Liouville function and its partial sums}

Let $\lambda:\mathbb{N}\mapsto \{-1,+1\}$ be the Liouville function defined as $\lambda(n)=1$  if $n$
is the product of an even number of primes, and $\lambda(n)=-1$ if it is the product of an odd number of primes,
counted with multiplicities. Put $L(x):=\sum_{n\leq x}\lambda(n)$ and $\widetilde{L}(x):=(L(x+0)+L(x-0))/2$.

Embarking on the relation
$$
\frac{\zeta(2s)}{\zeta(s)}=\prod_{p\in\mathcal{P}}\frac{1}{1+p^{-s}}=\prod_{p\in\mathcal{P}}\sum_{k\geq 0}\frac{(-1)^k}{p^{sk}}=\sum_{n\geq 1}\frac{\lambda(n)}{n^s},\quad \Re(s)>1
$$
and using again Perron's inversion with the subsequent residue decomposition one can check that under (RH)
$$
\widehat{L}(t):=\ee^{-t/2}\widetilde{L}(\ee^t)=\frac{1}{\zeta(1/2)}+\sum_{k\in\Za}\frac{\zeta(2\rho_k)}{\rho_k\zeta'(\rho_k)}\ee^{\ii \gamma_k t}+R_{L}(\ee^t),\quad t>0 %,
$$
for an appropriate remainder $R_L$ satisfying $R_L(x)\to 0$, as $x\to+\infty$, see, for example, Eq.~(4.21) in~\cite{Akbary+Ng+Shahabi:2014}. Like similar sums discussed earlier,
%As before,
the sum $L_{ap}(t):=\sum_{k\in\Za}\frac{\zeta(2\rho_k)}{\rho_k\zeta'(\rho_k)}\ee^{\ii \gamma_k t}$ is understood as the limit of
$\sum_{k\in\Za:|\gamma_k|\leq T}\frac{\zeta(2\rho_k)}{\rho_k\zeta'(\rho_k)}\ee^{\ii \gamma_k t}$ along an appropriate sequence $T=T_n$ diverging to infinity.

Invoking 
\begin{equation}\label{eq:zeta_on_the_line1}
\zeta(2\rho_k)=\zeta(1+2\gamma_k)=O((\log \gamma_k \log\log \gamma_k)^{3/4}),\quad k\to\infty,
\end{equation}
see Theorem 6.14 in~\cite{Titchmarsh:1986}, and the results of Section~\ref{sec:mobius} we conclude the following.

\begin{prop}\label{prop:Liouville}
Assume (RH) and~\eqref{eq:gonek}. As $L\to+\infty$, the random processes $L_{ap}(V_{0,L}+\cdot)$ converge in distribution on the space $B_2(\mathcal{Z}_{\zeta})$ (and, therefore, on the space $B_2$) to the stationary random process $\mathcal{M}_{L_{ap}}$. If, in addition,~\eqref{eq:assumption_li} holds true, then $(L_{ap}(V_{0,L}+t))_{t\in\mathbb{R}}$ converge to
$(\mathcal{M}_{L_{ap}}(t))_{t\in\mathbb{R}}$ also in the sense of finite-dimensional distributions.
\end{prop}

Under the assumptions of Proposition~\ref{prop:Liouville}, including~\eqref{eq:assumption_li}, the random Fourier series
$$
\mathcal{M}_{L_{ap}}(t)=2\Re\left(\sum_{k\geq 1}\frac{\zeta(2\rho_k)}{\rho_k\zeta'(\rho_k)}\mathcal{U}_{k}\ee^{\ii \gamma_k t}\right)
$$
converges a.s.~almost everywhere. Using~\eqref{eq:zeta_on_the_line1} and calculations similar to those in the proof of Proposition~\ref{prop:mertens_continuity_limit} we arrive at a proposition.
\begin{prop}
Assume (RH),~\eqref{eq:assumption_li} and~\eqref{eq:gonek}. The process $\mathcal{M}_{L_{ap}}$ is a.s.~$\gamma$-H\"{o}lder-continuous with $\gamma<1/2$.
\end{prop}

\subsection{Further examples}
Assuming (RH) and~\eqref{eq:assumption_li} it was shown in~\cite{Rubinstein+Sarnak:1994} that the remainders in the prime counting functions in arithmetic progressions possess limiting distributions. The results of the present paper can be used to derive functional versions of that result.  Many further examples of this flavor can be found in~\cite{Akbary+Ng+Shahabi:2014}.

\section*{Acknowledgments}

The work of AI was supported by the National Research Foundation of Ukraine (project number 2023.03/0059 `Contribution to modern theory of random series'). AM gratefully acknowledges financial support from the Alexander von Humboldt Foundation. ZK has been supported by the German Research Foundation under Germany’s Excellence Strategy EXC 2044~-~ 390685587, Mathematics M\"{u}nster: Dynamics - Geometry - Structure and by the DFG priority program SPP 2265 Random Geometric Systems.

\bibliographystyle{amsplain}
\bibliography{Biblio.bib}

\end{document}